\documentclass[12pt]{amsart}
\usepackage{amscd}

\usepackage{hyperref} 
\usepackage{amsmath}
\usepackage{amsfonts}
\usepackage{amssymb}
\usepackage{amsthm}
\usepackage{mathtools}
\usepackage[all,cmtip]{xy}

\usepackage[utf8]{inputenc}

\usepackage{pinlabel}
\usepackage{lmodern}
\usepackage{textcomp}
\usepackage{color}		

\newtheorem{thm}{Theorem}[section]

\newtheorem{defi}[thm]{Definition}

\newtheorem{prp}[thm]{Proposition}

\newtheorem{lma}[thm]{Lemma}

\newtheorem{rem}[thm]{Remark}
\newenvironment{rmk}{\begin{rem}\rm}{\end{rem}}

\numberwithin{equation}{section}

\makeatletter
\newcommand{\labitem}[2]{%
\def\@itemlabel{{#1}}
\item
\def\@currentlabel{#1}\label{#2}}
\makeatother

\newcommand{\R}{{\mathbb R}}
\newcommand{\Z}{{\mathbb Z}}
\newcommand{\C}{{\mathbb C}}

\newcommand{\Or}{\mathcal{O}}
\newcommand{\Co}{\mathcal{C}}
\newcommand{\Hi}{\mathcal{H}}
\newcommand{\M}{\mathcal{M}}

\newcommand{\A}{\mathcal{A}}

\newcommand{\E}{\mathcal{E}}

\newcommand{\Sy}{\mathcal{S}}

\newcommand{\dbar}{\bar{\partial}}
\newcommand{\id}{\operatorname{id}}
\newcommand{\ind}{\operatorname{index}}
\newcommand{\Ker}{\operatorname{Ker}}
\newcommand{\Coker}{\operatorname{Coker}}
\newcommand{\coker}{\operatorname{Coker}}
\newcommand{\kd}{\operatorname{Ker}\dbar}
\newcommand{\cd}{\operatorname{Coker}\dbar}

\newcommand{\diag}{\operatorname{Diag}}

\newcommand{\capp}{\operatorname{cap}}

\newcommand{\hbp}{\hat{\bigoplus}}

\begin{document}

\title{A note on coherent orientations for exact Lagrangian cobordisms}

\author{Cecilia Karlsson}
\address{Department of Mathematics,
	 University of Oslo,
         Postboks 1053,
         Blindern,
         0316 Oslo, Norway}
\email{cecikarl{\@@}math.uio.no}

\begin{abstract}
   Let $L \subset \R \times J^1(M)$ be a spin, exact Lagrangian cobordism
   in the symplectization of the 1-jet space of a smooth manifold
   $M$. Assume that $L$ has cylindrical Legendrian ends $\Lambda_\pm
   \subset J^1(M)$. It is well known that the Legendrian contact
   homology of 
   $\Lambda_\pm$ can be defined with integer coefficients, via a signed
   count of pseudo-holomorphic disks in the cotangent bundle of $M$. It is also known that this count can be lifted to a mod 2 count of pseudo-holomorphic disks in the symplectization $\R \times J^1(M)$, and that
   $L$ induces a morphism between the $\Z_2$-valued DGA:s of the ends $\Lambda_\pm$ in a
   functorial way. We prove that this hold with integer coefficients as well.
   
   The proofs are built on the technique of orienting the moduli spaces of
   pseudo-holomorphic disks using
   capping operators at the Reeb chords. We give an expression for how the DGA:s   change if we change the capping operators.
   \end{abstract}

\thanks{The author was supported by the grant KAW 2015.0353 from the Knut and Alice Wallenberg foundation and the ERC grant Geodycon.}

\maketitle

\section{Introduction}

\subsection{Background}
Let $M$ be an $n$-dimensional manifold and consider the $1$-jet space
$J^1(M)= T^*M \times \R$ of $M$. This space can be given the structure
of a contact manifold, with contact form  $\alpha= dz - \sum_j y_j
dx_j$. Here $(x,y)$ are coordinates on $T^*M$ and $z$ is the
coordinate in the $\R$-direction. An $n$-dimensional submanifold
$\Lambda \subset J^1(M)$ is called \emph{Legendrian} if it is
everywhere tangent to the contact distribution $\xi = \Ker \alpha$,
and a \emph{Legendrian isotopy} is a smooth 1-parameter family of
Legendrian submanifolds. A major problem in contact geometry is to
determine whether two given Legendrian
submanifolds are \emph{Legendrian isotopic}, i.\ e. if there is a
Legendrian isotopy connecting them.
To that end, a number of \emph{Legendrian invariants} have
been introduced. These are objects associated to Legendrian
submanifolds, invariant under Legendrian isotopies.

One such invariant
is \emph{Legendrian contact homology}, which is the homology of a
differential graded algebra (DGA) associated to the Legendrian $\Lambda$. This
algebra is called the \emph{Chekanov-Eliashberg algebra of $\Lambda$},
and we denote it by $\A(\Lambda)$. It is a free, unital algebra
generated by the \emph{Reeb chords} of
$\Lambda$, which are flow segments of the \emph{Reeb vector field}
$\partial_z$, having their start and end points on $\Lambda$. We assume
that $\Lambda$ is chord generic, meaning that the \emph{Lagrangian
  projection} $\Pi_\C: J^1(M) \to T^*M$ projects the Reeb chords of
$L$ to isolated double points of $\Pi_\C(\Lambda)$. The differential
of the DGA is defined by counting certain pseudo-holomorphic disks.

Legendrian contact homology fits into the machinery of
\emph{Symplectic field theory}, introduced by Eliashberg,
Givental and Hofer in \cite{sft}. In particular, let $L$
be an exact Lagrangian cobordism in the \emph{symplectization} $ (\R \times
J^1(M), d(e^t\alpha))$ of $J^1(M)$. Assume that $L$ is asymptotic to
cylinders $\R \times \Lambda_\pm$ at $\pm \infty$, where
$\Lambda_\pm\subset J^1(M)$
are Legendrians. According to \cite{ratsft}, if we choose the coefficient ring to be given by $\Z_2$, then $L$ induces a DGA-morphism $\Phi_L:\A(\Lambda_+)
\to \A(\Lambda_-)$ in a functorial way. Here $\Phi_L$ is defined via a modulo 2 count of
pseudo-holomorphic disks with boundary on $L$. This is used in \cite{tobhonda} to
derive results about isotopy classes of exact Lagrangians with
prescribed boundary. More precisely, these results were derived
 from  explicit descriptions of $\Phi_L$ in the case
$L$ is induced by the trace of an elementary Legendrian isotopy.

That Legendrian contact homology can be defined over $\Z$, provided $\Lambda$ is spin, is proven in
\cite{orientbok}. In that paper the differential of $\A(\Lambda)$ is
defined by a count of rigid pseudo-holomorphic disks in the Lagrangian
projection,  with the disks having boundary
on $\Pi_\C(\Lambda)$. To get a signed count of these disks, it is shown that the moduli space of pseudo-holomorphic disks admits a coherent orientation. However, there is another way to define the
differential, which is more convenient if one wants to consider the
functorial properties in Symplectic field theory.  That method is to
count rigid pseudo-holomorphic disks in the symplectization
of $J^1(M)$, with the disks having boundary on $\R \times \Lambda$.

In
\cite{Georgios} it is proven that
these two different counts give the same
DGA, given that we
work with $\Z_2$-coefficients. We will prove that this also holds with
$\Z$-coefficients, provided that $\Lambda$ is spin.
More precisely, we will prove that the coherent orientation scheme
given in \cite{orientbok} can be lifted to give a coherent orientation for moduli spaces of pseudo-holomorphic disks in $\R \times
J^1(M)$ with boundary on $\R \times \Lambda$.
Then we prove that this lifted  orientation scheme allows us to extend
the definition of $\Phi_L$ from \cite{ratsft} to $\Z$-coefficients, provided that $L$ is spin and that
$\Lambda_\pm$ are given the induced spin structure as boundary of
$L$. That this lift can be performed seems in particular important if
one wants to relate SFT theories with Floer theories, for example via
Seidel's isomorphisms which briefly says that if $\Lambda$ admits an
exact Lagrangian filling $L$, then there is an isomorphism between
$H_*(L)$ and the
linearized Legendrian contact cohomology of $\Lambda$ with respect to
the augmentation induced by $L$. Compare with \cite{Georgios} and \cite{georgios-roman}.  
For other examples of applications of such a signed lift which allows $\Phi_L$ to be defined over the integers, see e.g.\ \cite{floerlag}, \cite{floerlag2}, \cite{toblose}, \cite{el}. Note that the existence of such a signed lift is indicated but not proved in these papers. 
Yet another motivation for understanding DGA-morphisms with coefficients in $\Z$ comes from the connection between Legendrian contact homology and homological mirror symmetry, together with the machinery in \cite{rm}. 
%

The coherent orientation scheme for the moduli spaces of pseudo-holomorphic
disks will be defined by using something called \emph{capping
  operators}, which are $\dbar$-operators defined on the 1-punctured
unit disk in $\C$ with trivialized Lagrangian boundary
conditions. Using the DGA-morphism  induced by the trivial cobordism
$\R \times \Lambda$, we will derive an expression of how the DGA
changes if we change capping operators. In this way we can relate the
 orientation scheme of pseudo-holomorphic disks in $T^*M$ given in
 \cite{orientbok} with the one given in
 \cite{orienttrees}.

 The orientation scheme defined in \cite{orienttrees} is adapted to the situation
 when the differential in Legendrian contact homology is defined by
 counting rigid \emph{Morse flow trees} instead of pseudo-holomorphic
 disks. We refer to \cite{trees} for the definition of these trees,
 and for the proof that the trees can replace the pseudo-holomorphic disks
 in the definition of the differential if we work with
 $\Z_2$-coefficients. In \cite{orienttrees}, this result is extended
 to also hold for $\Z$-coefficients. The advantage of using Morse flow trees instead of 
 pseudo-holomorphic disks is that the former ones can be found using
 finite-dimensional flow techniques, while the latter ones give rise
 to non-linear PDE:s, which in general are hard to solve.
 In \cite{tobhonda} it is shown that one can use Morse flow trees to
compute the DGA-morphism induced by an exact Lagrangian cobordism, in
the case when the coefficients are given by $\Z_2$. This is one of the
reasons why the DGA-morphisms induced by traces of
elementary Legendrian isotopies can be described explicitly when $n=1$. We sketch an argument that Morse flow trees can be used to compute DGA-morphisms also with integer coefficients, given our orientation scheme of moduli spaces.


\subsection{Organization of the paper}
In Section \ref{sec:mainresults} we give a definition of
the DGA associated to a Legendrian $\Lambda \subset J^1(M)$, and the
DGA-morphism induced by an exact Lagrangian cobordism. We also state
the main theorems.
In Section \ref{sec:holo} we recall the definition of punctured
pseudo-holomorphic disks, and give a more detailed definition of the
relevant moduli spaces.  In Section \ref{sec:oc} we fix orientation
conventions, and prove that these conventions make it possible to define Legendrian contact homology with integer coefficients in the symplectization setting. In  Section \ref{sec:cob-pf} we prove that this also gives the desired results for the DGA-morphisms induced by exact Lagrangian cobordisms. In Section \ref{sec:trees} we discuss how the
orientation scheme can be used to orient the moduli space of
Morse flow trees associated to exact Lagrangian cobordisms.

\subsection*{Acknowledgments}
This paper is built on parts of the author's PhD thesis at Uppsala University. The work was further developed when the author was a postdoc at the University of Nantes, and completed while the author was a postdoc at Stanford University.

The author would like to thank  Tobias Ekholm and Paolo Ghiggini for useful discussions.

\section{Main results}
\label{sec:mainresults}
Here we formulate the main results. To be able to do this, we first
need to introduce some more notation.

\subsection{Legendrian contact homology}
As outlined in the Introduction, there are two different ways of defining the differential $\partial$ of $\A(\Lambda)$. One method is to compute punctured, rigid, pseudo-holomorphic disks in $T^*M$ with boundary on  $\Pi_\C(\Lambda)$. I.e., the differential is defined on generators $a$ by
\begin{equation*}
 \partial_l(a) = \sum_{\dim \M_{l,\Lambda}(a,{\bf b}) = 0} |\M_{l,\Lambda}(a, {\bf b})| {\bf b},
\end{equation*}
and extended by the Leibniz rule to the rest of the algebra. Here ${\bf b} = b_1\dotsm b_m$ is a word of Reeb chords, $\M_{l,\Lambda}(a,{\bf b})$ is the moduli space of rigid pseudo-holomorphic punctured disks with a positive puncture at $a$, negative punctures at $b_1,\dotsc,b_m$, and with boundary on $\Pi_{\C}(\Lambda)$, and $|\M_{l,\Lambda}(a, {\bf b})|$ denotes the algebraic count of disks in the moduli space.  We refer to Section \ref{sec:holo} for more details.

We denote the DGA defined in this way by $(\A(\Lambda), \partial_l; R)$, where $R$ indicates the coefficient ring.  In \cite{pr}  it is  proven that for a generic choice of compatible almost complex structure on  $T^*M$, this differential satisfies $\partial_l^2 =0$, and the homology of this complex gives a well-defined Legendrian isotopy invariant if we choose the coefficient ring to be $\Z_2$. In \cite{orientbok} these results were extended to hold for $\Z$-coefficients in the case when $\Lambda$ is spin. In the special case $n=1$ and $M=\R$, these results were first established in \cite{chekanov} for the case of $\Z_2$-coefficients, and in \cite{sabloffng} for $\Z$-coefficients.

The other method of computing the differential, which was discussed in \cite{sft}, and where the details were worked out in \cite{eliinv} for $n=1$, and further developed in \cite{ratsft} for higher dimensions, is to count rigid pseudo-holomorphic disks in the symplectization of $J^1(M)$. That is, in this case the differential is defined by
\begin{equation*}
 \partial_s(a) = \sum_{\dim \hat\M_{s,\Lambda}(a,{\bf b}) = 1} |\hat\M_{s,\Lambda}(a, {\bf b})/\R| {\bf b}
\end{equation*}
on generators, and again extended by the Leibniz rule to the whole
algebra. Here $\hat\M_{s,\Lambda}(a,{\bf b})$ is the moduli space of
punctured pseudo-holomorphic disks with boundary on $\R \times \Lambda$,
having a positive puncture asymptotic to a strip over the Reeb chord
$a$ at $t=+ \infty$, and having negative punctures asymptotic to
strips over the Reeb chords $b_1,\dotsc,b_m$ at $t=-\infty$. Moreover,
we assume that the given almost complex structure is cylindrical, so that we get an induced $\R$-action on $\hat
\M_{s,\Lambda}$, given by translation in the $t$-direction. See Section
\ref{sec:holo}.
We let 
\begin{equation*}
 \M_{s,\Lambda}=\hat\M_{s,\Lambda}(a, {\bf b})/\R
\end{equation*}
be the space where we have divided out this $\R$-action.  For a generic choice of cylindrical almost complex structure we have $\partial_s^2=0$, given that we are using $\Z_2$-coefficients, 
and the homology of $\A(\Lambda)$ is invariant under Legendrian isotopies. See \cite{ratsft}.

In \cite{Georgios} it is shown that under certain, not too restrictive, choices of almost complex structures of $T^*M$ and $\R \times J^1(M)$ we have that 
 $(\A(\Lambda), \partial_s;\Z_2) \simeq (\A(\Lambda), \partial_l;
 \Z_2)$, 
where the isomorphism is induced by the projection
\begin{equation*}
 \pi_P: \R \times (T^*M \times \R) \to T^*M.
\end{equation*}
In particular, it is proven that the induced map
\begin{align*}
\pi_P: \M_{s, \Lambda}(a, {\bf b}) \to \M_{l,\Lambda}(a, {\bf b}),
                                     \quad  u \mapsto  \pi_P(u),
\end{align*}
is a diffeomorphism. In the present paper we extend this result to $\Z$-coefficients, 
by proving that there is a choice of orientation conventions so that
the coherent orientation scheme given for $\M_{l,\Lambda}$ in
\cite{orientbok} can be lifted under $\pi_P$ to give a coherent
orientation scheme for $\M_{s,\Lambda}$. Compare [\cite{Georgios}, Remark 2.4].

\begin{thm}\label{thm:georgios}
 Let $J_P$ and $\tilde J_p$ be almost complex structures on $T^*M$ and $\R \times J^1(M)$, respectively, satisfying the assumptions in [\cite{Georgios}, Theorem 2.1]\footnote{The conditions are that $(D\pi_p)\tilde J_p = J_p(D\pi_p)$, and that $J_p$ is regular (the 0-dimensional moduli spaces $\M_{l,\Lambda}$ are transversely cut out) and integrable in neighborhoods of the double points of $\Pi_\C(\Lambda)$.}. Further assume that $\Lambda \subset J^1(M)$ is a spin Legendrian submanifold. Then there are choices of coherent orientations of the moduli spaces $\M_{s,\Lambda}(a, {\bf b})$  and $\M_{l,\Lambda}(a, {\bf b})$ so that 
 \begin{align*}
\pi_P: \M_{s, \Lambda}(a, {\bf b}) \to \M_{l,\Lambda}(a, {\bf b}),
                                     \quad u \mapsto  \pi_P(u),
\end{align*}
is orientation preserving. Moreover, for $i=s,l$ we have that 
\begin{equation}\label{eq:differential}
 \partial_i(a) = \sum_{\dim \M_{i,\Lambda}(a,{\bf b}) = 0} |\M_{i,\Lambda}(a, {\bf b})| {\bf b}
\end{equation}
 satisfies 
 $\partial_i^2 =0$. Here $ |\M_{i,\Lambda}(a, {\bf b})| $ denotes the
 algebraic number of disks in the moduli space, where the signs of the
 disks come from the coherent orientation scheme.
\end{thm}

\begin{rmk}
  We also get that the stable tame isomorphism class of the DGA:s is
  invariant under Legendrian isotopies. Compare [\cite{orientbok},
  Section 4.3].
\end{rmk}

\begin{rmk}\label{rmk:dumsign}
  We will use slightly different orientation conventions than in 
 \cite{orientbok}, to simplify the expression of the differential. In
 that paper it is instead of \eqref{eq:differential} given by
 \begin{equation*}
 \partial_l(a) = \sum_{\dim \M_{l,\Lambda}(a,{\bf b}) = 0} (-1)^{(n-1)(|a|+1)}|\M_{l,\Lambda}(a, {\bf b})| {\bf b}.
\end{equation*}
Compare with the discussion in Subsection \ref{sec:clock}.
\end{rmk}

\subsection{Exact Lagrangian cobordisms}\label{sec:elc}
Here we describe how an exact Lagrangian
submanifold $L \subset \R \times J^1(M)$ with cylindrical Legendrian
ends induces a morphism between the DGA:s of the ends. 

\begin{defi}\label{def:exactcob}
Let $\Lambda_+, \Lambda_- \subset J^1(M)$ be Legendrian submanifolds.  
An \emph{exact Lagrangian cobordism from $\Lambda_+$ to $\Lambda_-$} is an exact Lagrangian submanifold $L$ of the symplectization of $J^1(M)$, satisfying
\begin{align*}
& \E_+(L) := L \cap ((T,\infty) \times J^1(M)) = (T, \infty) \times \Lambda_+, \\
& \E_-(L) := L \cap ((-\infty, -T) \times J^1(M)) = (-\infty, -T) \times \Lambda_-, 
\end{align*}
for some $T>0$, and so that 
\begin{enumerate}
 \item each function $f$ that satisfies $d f = e^t \alpha|_L$, also satisfies that $f|_{\E_\pm(L)}$ is constant,
 \item $L \setminus (\E_+(L) \cup \E_-(L))$  is compact with boundary $\Lambda_+ - \Lambda_-$.
\end{enumerate}
\end{defi}

An exact Lagrangian cobordism $L$ induces a DGA-morphism 
\begin{equation*}
\Phi_L:(\A(\Lambda_+), \partial_+; \Z_2) \to   (\A(\Lambda_-), \partial_-; \Z_2),\end{equation*}
where $\partial_{\pm}$ denotes the differential $\partial_s$ associated to $\A(\Lambda_{\pm})$.
Indeed, we can define $\Phi_L$  by 
 \begin{equation}\label{eq:phidisk}
  \Phi_L(a) = 
  \sum_{\dim \M_L (a, {\bf b})=0} |\M_L(a, {\bf b})| {\bf b},  \qquad {{\bf b} = b_1\dotsm b_m, } 
 \end{equation}
if $a$ is a generator, and extend it to the rest of the algebra by
\begin{align}
 \label{eq:ex1}  &\Phi(a +b ) = \Phi(a) + \Phi(b) \\
 \label{eq:ex2} &\Phi(ab) = \Phi(a)\Phi(b).
\end{align}
See  \cite{ratsft} and [\cite{tobhonda}, Section 3.5].
Here  $\M_L(a,{\bf b})$ denotes the moduli space of punctured pseudo-holomorphic disks with boundary on $L$, positive puncture mapped asymptotically to a strip over the Reeb chord $a$ at  $t = + \infty$, negative punctures mapped asymptotically to strips over the Reeb chords $b_1,\dotsc,b_m$ at $t=-\infty$, and $|\M_L(a, {\bf b})|$ is the modulo $2$ count of elements. See Section \ref{sec:holo}. Note that in [\cite{tobhonda}, Section 3.5] the results are only stated for $n=1$, but tracing the proofs one sees that they can be extended word-by-word to arbitrary $n$.

We prove that we can replace the modulo 2 count by a signed count, so that $\Phi_L$ gives a DGA-morphism also with $\Z$-coefficients.

\begin{thm}\label{thm:kalman}
Let  $L\subset \R \times J^1(M)$ be a spin, exact Lagrangian cobordism from $\Lambda_+$ to $\Lambda_-$. Then there are choices of coherent orientations of the moduli spaces $\M_{L}(a,{\bf b})$,  $\M_{s,\Lambda_+}(a,{\bf b})$ and $\M_{s,\Lambda_-}(a,{\bf b})$ so that  
 \begin{equation*}
   \Phi_L:(\A(\Lambda_+), \partial_+; \Z) \to   (\A(\Lambda_-), \partial_-; \Z)                                                                                                                                                                    \end{equation*}
 defined by \eqref{eq:phidisk} -- \eqref{eq:ex2} is a DGA-morphism. Now $|\M_L(a, {\bf b})|$ represents the algebraic count of disks in the moduli space. 
\end{thm}

Moreover, $\Phi_L$ satisfies SFT-functorality. That is, let $L_1$, $L_2 \subset \R \times J^1(M)$ be two exact Lagrangian cobordisms such that $L_1$ goes from $\Lambda_0$ to $\Lambda_1$ and $L_2$ goes from $\Lambda_1$ to $\Lambda_2$. Then we can form the \emph{concatenation} $L_1 \# L_2$, by gluing the negative end of $L_1$ to the positive end of $L_2$, as explained in [\cite{tobhonda}, Section 1.2].  This gives an exact Lagrangian cobordism from $\Lambda_0$ to $\Lambda_2$, which satisfies 
   $  \Phi_{L_1\#L_2} = \Phi_{L_2} \circ \Phi_{L_1}$
  as a DGA-morphism from $(\A(\Lambda_0), \Z_2)$ to $(\A(\Lambda_2), \Z_2)$. See [\cite{tobhonda}, Lemma 3.13]. 

We prove that the functorial properties of $\Phi$ continue to hold with integer coefficients.
\begin{thm}\label{thm:functorial}
  Assume that $L_1,L_2 \subset \R \times J^1(M)$ are two spin, exact Lagrangian cobordisms with fixed spin structures. Assume that $L_1$ goes from $\Lambda_0$ to $\Lambda_1$ and that $L_2$ goes from $\Lambda_1$ to $\Lambda_2$. Then there are choices of coherent orientations of the moduli spaces $\M_{s,\Lambda_i}(a, {\bf b})$, $i=0,1,2$, $\M_{L_i}(a, {\bf b})$, $i=1,2$, and $\M_{L_1\#L_2}(a, {\bf b})$ so that 
  \begin{equation}
    \label{eq:funct1}
    \Phi_{L_2} \circ \Phi_{L_1} = \Phi_{L_1\#L_2}
  \end{equation}
  as DGA-morphisms from $(\A(\Lambda_0), \Z)$ to $(\A(\Lambda_2), \Z)$.
  Moreover, if $\Lambda\subset J^1(M)$ is a spin Legendrian then there are choices of coherent orientations of the moduli spaces $\M_{s,\Lambda}(a, {\bf b})$, $\M_{\R \times \Lambda}(a, {\bf b})$ so that 
  \begin{equation}
    \label{eq:funct2}
    \Phi_{\R \times \Lambda} = \id.
  \end{equation}
\end{thm}

We will prove that the orientation scheme from Theorem
\ref{thm:georgios} can be used to derive these results.
As indicated in the
Introduction, this orientation scheme is defined using capping
operators. Briefly, this works as follows.

Let $u \in \M_{l,\Lambda}(a, {\bf b})$. We have a linearized $\dbar$-operator $\dbar_u$ associated to $u$, defined on the punctured unit disk in $\C$ and with a trivialized Lagrangian boundary condition induced by the spin structure of $\Lambda$. This boundary condition is ``closed up'' by
gluing capping disks to the punctures of $u$. That is, for each Reeb
chord $c$ of $\Lambda$ we define two different capping operators
$\dbar_{c,+}$ and $\dbar_{c,-}$. These are $\dbar$-operators defined on
the unit disk in $\C$ with one puncture. We glue $\dbar_{a,+}$ to $\dbar_u$ at the positive
puncture of $u$, and $\dbar_{b_{i},-}$ to $\dbar_u$ at the negative puncture
corresponding to the chord $b_i$, $i=1,\dotsc , m$. We require the
trivialized boundary conditions for the capping operators to be
defined in such a way so that these gluings induce a trivialized Lagrangian  boundary
condition for the non-punctured unit disk in $\C$.  Then we use the
fact that there is a  canonical orientation of the determinant line
bundle for the $\dbar$-operator over the space of trivialized
Lagrangian boundary conditions for the unit disk in $\C$. This
canonical orientation is given via evaluation at the boundary, see
[\cite{fooo}, Section 8], and the canonical orientation for the $\dbar$-operator
associated to the capped boundary condition induces
an orientation of the determinant line $\det \dbar_u$, which in turn induces an orientation of $T_u \M_{l,\Lambda}(a, {\bf b})$. This is explained in more detail in
Section \ref{sec:holo} and Section \ref{sec:oc}.

Notice that the signs occurring in the differential of the DGA of $\Lambda$ depend on the
choice of capping operators. We  will prove that for certain systems
of capping operators, the associated DGA:s are isomorphic.

\begin{thm}
  \label{thm:differentcappings}
  Let $\Lambda$ be a spin Legendrian submanifold of $J^1(M)$ with a fixed spin structure. Let $\mathcal{S}$ denote a system of capping operators for $\Lambda$
  satisfying \ref{c1} -- \ref{c3} in Section \ref{sec:orientcap}. Let $\partial_{l,\mathcal{S}}$ denote the induced differential as defined in \eqref{eq:differential}, where the orientation of the moduli space is induced by the system $\mathcal{S}$. Then $(\A(\Lambda), \partial_{l,\mathcal{S}};\Z)$ is a DGA whose homology is invariant under Legendrian isotopies.

  Moreover, if $\mathcal{S'}$ is another system of capping operators for $\Lambda$ satisfying  \ref{c1} -- \ref{c3}, then there is a DGA-isomorphism
  \begin{equation}
    \label{eq:isocap}
    \Phi_{\mathcal{S},\mathcal{S}'}:(\A(\Lambda), \partial_{l,\mathcal{S}};\Z) \to (\A(\Lambda), \partial_{l,\mathcal{S}'};\Z).
  \end{equation}
\end{thm}

We refer to Section \ref{sec:cob-pf} for an explicit description of the map (\ref{eq:isocap}).

\begin{rmk}
  From the proofs of Theorem \ref{thm:georgios}, Theorem \ref{thm:kalman} and Theorem \ref{thm:functorial} it follows that any system of capping operators satisfying  \ref{c1} -- \ref{c3}  gives coherent orientation schemes so that the statements of the  theorems hold.
\end{rmk}

\begin{rmk}
 Note that the capping operators defined in [\cite{orientbok}, Section 3.3], [\cite{orientbok}, Section 4.5] and [\cite{orienttrees}, Section 3.4] all satisfy \ref{c1} -- \ref{c3}. Compare with Remark \ref{rmk:notv}.
\end{rmk}

\begin{rmk}\label{rmk:oras}
  All orientation schemes above depend on choices of orientations of $\R^n$ and of $\C$, which we from now on assume to be fixed.
\end{rmk}

\section{Punctured pseudo-holomorphic disks}\label{sec:holo}

In this section we give a definition of punctured pseudo-holomorphic disks. We also define the moduli spaces that will be relevant for us.

\subsection{Pseudo-holomorphic disks}
 An \emph{almost complex structure} $J$ on a symplectic manifold $(X, \omega)$ is  an endomorphism $J: TX \to TX$ satisfying $J^2 =-\id$. We say that $J$ is \emph{compatible with $\omega$} if $\omega(\cdot, J \cdot)$ defines a Riemannian metric on $X$. If $(X, \omega) = (\R \times J^1(M), d(e^t \alpha))$, then $J$ is \emph{cylindrical} if it is compatible with $\omega$, is invariant under $t$-translation, and satisfies $J(\xi)= \xi$, $J(\partial_t) = R_\alpha$. Here $R_\alpha$ denotes the Reeb vector field of $\alpha$.  

Let $D$ be the compact unit disk in $\C$ and let $D_{m+1}$ denote the disk with $m+1$ marked points $p_0,\dotsc,p_m \in \partial D$, cyclically ordered along the boundary in the counter-clockwise direction. Let $\dot D_{m+1}$ denote the corresponding punctured disk with the marked points removed. 
We will assume that  $p_0=1 \in \C$, and call it the \emph{positive puncture}. We say that  $p_1,\dotsc,p_m$ are the \emph{negative punctures}.

A map $u:D_{m+1} \to X$ (or $u:\dot D_{m+1} \to X$ ) is \emph{J-holomorphic}  if it satisfies
\begin{equation*}
 \dbar_J(u) := du + J \circ du \circ i =0. 
\end{equation*}
If we want to neglect the choice of $J$ we say that $u$ is \emph{pseudo-holomorphic}.

\subsection{Gradings}
Each Reeb chord $a$ of $\Lambda$ comes equipped with a grading $|a|$, given by 
\begin{equation*}
 |a| = CZ(a) - 1
\end{equation*}
where $CZ(a)$ is
the \emph{Conley-Zehnder index} of $a$. Since we will not perform any explicit calculations of the gradings in this paper we refer to [\cite{pr}, Section 2.2] for a proper definition. 

\subsection{Moduli spaces}
In this section we give definitions of the relevant moduli spaces of pseudo-holomorphic disks. 

\subsubsection{Moduli spaces in the Lagrangian projection}
Fix an almost complex structure $J$ on $T^*M$, compatible with $\omega$. We let $\M_{l,\Lambda}(a, \bf b)$, ${\bf b}=b_1\dotsm b_m$, denote the moduli space of pseudo-holomorphic maps $u:(D_{m+1}, \partial D_{m+1}) \to (T^*M, \Pi_\C(\Lambda))$ satisfying the following:
\begin{enumerate}
 \item $u|_{\partial \dot D_{m+1}}$ has a continuous lift $\tilde u$ to $\Lambda$; 
 \item $u(p_0)= \Pi_\C(a)$, where $a$ is a Reeb chord of $\Lambda$, and the $z$-coordinate of $\tilde u$ makes a positive jump when passing through $p_0$ in the counterclockwise direction;
 \item $u(p_i)= \Pi_\C(b_i)$,  $i=1,\dotsc,m$, where $b_i$ is a Reeb chord of $\Lambda$,  and the $z$-coordinate of $\tilde u$ makes a negative jump when passing through $p_i$ in the counterclockwise direction.
\end{enumerate}

Moreover, we consider two maps $u_1$, $u_2$ satisfying the above to be equal if they differ by a biholomorphism of $D_{m+1}$. 

In \cite{legsub} it is proven that for generic $J$ the moduli spaces are transversely cut out manifolds
of dimension
\begin{equation}\label{eq:dim1}
\dim\M_{l,\Lambda}(a,{\bf b}) = |a| - \sum_{i=1}^m |b_i| -1.
 \end{equation}

\subsubsection{Moduli spaces in the symplectization}
Fix a cylindrical almost complex structure $J$ on $\R \times J^1(M)$. We let $\hat\M_{s,\Lambda}(a, \bf b)$ denote the moduli space of pseudo-holomorphic maps $u:(\dot D_{m+1}, \partial \dot D_{m+1}) \to (\R \times J^1(M), \R \times \Lambda)$ satisfying the following:
\begin{enumerate}
 \labitem{{(s1)}}{s1} in a neighborhood of the positive puncture $p_0$ the map $u$ is asymptotic to the Reeb chord strip $[0,\infty) \times c$; 
 \labitem{{(s2)}}{s2} in a neighborhood of the negative puncture $p_i$ the map $u$ is asymptotic to the Reeb chord strip $(-\infty,0] \times b_i$ $i=1,\dotsc,m$. 
\end{enumerate}

Again, we consider two maps $u_1$, $u_2$ satisfying the above to be equal if they differ by a biholomorphism of $D_{m+1}$. 

We let 
$\M_{s,\Lambda}(a, {\bf b}) = \hat \M_{s,\Lambda}(a, \bf b)/\R$
where the $\R$-action is given by translation in the $t$-direction.

For generic $J$ the moduli spaces are transversely cut out 
manifolds of dimension
\begin{equation*}
\dim\M_{s,\Lambda}(a,{\bf b}) = |a| - \sum_{i=1}^m |b_i|-1.
\end{equation*}
See [\cite{Georgios}, Section 4.2.4].

\subsubsection{Moduli spaces associated to an exact Lagrangian cobordism}
 Fix a compatible almost complex structure $J$ on $\R \times J^1(M)$, and assume that it is cylindrical for $|t|>N$ for some $N$. We let $\M_{L}(a, \bf b)$ denote the moduli space of pseudo-holomorphic maps $u:(\dot D_{m+1}, \partial \dot D_{m+1}) \to (\R \times J^1(M), L)$ satisfying \ref{s1} and \ref{s2},
and again we consider two maps $u_1$, $u_2$ to be equal if they differ by a biholomorphism of $D_{m+1}$. 

For generic $J$ the moduli spaces are transversely cut out manifolds of dimension 
\begin{equation}\label{eq:dim2}
\dim\M_{L}(a,{\bf b}) = |a| - \sum_{i=1}^m |b_i|.
\end{equation}
See [\cite{tobhonda}, Lemma 3.7]. 

From now on we assume that the almost complex structures are chosen so that the relevant moduli spaces are transversely cut out manifolds of the expected dimension.
We call a disk $u \in \M_{i,\Lambda}(a,{\bf b})$, $i=l,s$, a \emph{pseudo-holomorphic disk of $\Lambda$ with positive puncture $a$ and negative punctures $b_1,\dotsc,b_m$}. If moreover  $\dim \M_{i,\Lambda}(a,{\bf b}) =0$  we say that $u$ is \emph{rigid}. We use similar language for disks $u \in  \M_{L}(a,{\bf b})$.

\subsection{The linearized \texorpdfstring{$\dbar$}{dbar}-operator}
The algebraic count of elements in \eqref{eq:differential} and \eqref{eq:phidisk} are defined by associating a sign to each rigid pseudo-holomorphic disk. This assignment of signs can be understood as an orientation of the moduli spaces, and since these spaces are assumed to be zero-dimensional they are always orientable.  However, since we require that  $\partial^2= 0$ and $\partial \circ \Phi_L = \Phi_L \circ \partial$, we need to choose the orientations in a coherent way. This is done by considering linearized $\dbar$-operators associated to the pseudo-holomorphic disks.

Let the Sobolev space $\Hi_k({D}_{m+1}, \C_n)$ be the closure of $\Co_0^\infty (\dot{D}_{m+1}, \C^n )$ equipped with the 
standard Sobolev $\|\cdot \|_{k,2}$-norm. That is, $\Hi_k({D}_{m+1}, \C_n)$ consists of all elements in $L^2({D}_{m+1}, \C_n)$ whose weak derivatives exist and belong to $L^2$, up to order $k$. 

Choose local coordinates on $\dot D_{m+1}$ in a neighborhood of the puncture $p_j$, given by a half-infinite strip
$E_{p_j} = {(\tau, t) \in [0,\infty) \times [0, 1]},
(\tau, t) = \tau + it$. For each puncture $q_i \in D_{m+1}$ we define a \emph{weight vector} 
\begin{equation*}
 \nu_i = (\nu_i^1,\dotsc,\nu_i^n) \in (-\pi/2, \pi/2)^n
\end{equation*}
and let $\nu= (\nu_0,\dotsc,\nu_m)$. Let $w_{\nu}: D_{m+1} \to GL(n)$ be a smooth function satisfying
  \begin{align*}
    w_{\nu}(\tau,t) =
    \diag(e^{\nu_i^1 |\tau|},\dotsc,e^{\nu_i^n |\tau|})
  \end{align*}
  in $E_{p_i}(M)$, and assume that $w_{\nu}$ is close to the identity
  matrix in compact regions of the disk. 
  
  Let the \emph{weighted Sobolev space $\Hi_{k,\nu}(D_{m+1},\C^n)$} be defined by 
\begin{equation*}
 \Hi_{k,\nu}(D_{m+1},\C^n) = \{f \in \Hi_k^{loc}(D_{m+1},\C^n); w_{\nu}f \in \Hi_k(D_{m+1},\C^n)\}
\end{equation*}
with norm 
\begin{equation*}
\|f\|_{k,\nu} = \|w_{\nu}f\|_{k,2}. 
\end{equation*}

If $u: (D_{m+1}, \partial D_{m+1}) \to (T^*M, \Pi_\C(\Lambda))$ is a pseudo-holomorphic disk, then $u^* T\Pi_\C(\Lambda)$ induces a Lagrangian boundary condition on $D_{m+1}$. Pick a complex trivialization of $u^*TT^*M$. Using that $\Lambda$ is spin, we get a well-defined trivialization of the Lagrangian boundary condition, following [\cite{orientbok}, Section 3.4.2] and [\cite{pr}, Section 4.4]. This gives a collection of maps 
 $A = (A_0,\dotsc,A_{m+1}): \partial D_{m+1} \to U(n),
$ 
where 
\begin{equation*}
 A_i : [p_i, p_{i+1}] \to U(n), \qquad i=0,\dotsc,m+1, \quad m+2 =0.
\end{equation*}

Let 
\begin{equation*}
\Hi_{2,\nu}[A](D_{m+1}, u^*TT^*M) 
 \end{equation*}
 denote the closed subspace of  $\Hi_{2,\nu}(D_{m+1}, u^*TT^*M)$, consisting of elements $s$ that satisfy the linearized Lagrangian boundary condition $A$ along $\partial D_{m+1}$, and which satisfy 
 $\dbar_u s |_{\partial D_{m+1}} = 0$. 
Similarly, let 
\begin{equation*}
\Hi_{1,\nu}[0](D_{m+1}, T^{*0,1}D_{m+1} \otimes  u^*TT^*M)                                                                         \end{equation*}
 be the closed subspace of  $\Hi_{1,\nu}(D_{m+1}, T^{*0,1}D_{m+1} \otimes  u^*TT^*M)$ consisting of elements $s$ satisfying $ s |_{\partial D_{m+1}} = 0$.

From this we get an associated linearized $\dbar$-operator
\begin{equation*}
 \dbar_A = \dbar_{l,A}: \Hi_{2,\nu}[A](D_{m+1}, u^*TT^*M) \to \Hi_{1,\nu}[0](D_{m+1}, T^{*0,1}D_{m+1} \otimes  u^*TT^*M).
\end{equation*}

The boundary condition $A$ lifts to a boundary condition $\id \oplus A$ under $\pi_P$, and gives rise to a similar operator
\begin{multline*}
\dbar_{s,A}:\Hi_{2,\nu}[\id \oplus A](D_{m+1},\tilde u^*T(\R \times J^1(M))) \to\\
\Hi_{1,\nu}[0](D_{m+1},T^{*0,1}D_{m+1} \otimes \tilde u^*T(\R \times J^1(M))),
\end{multline*}
where $\tilde u$ is the lift of $u$ under $\pi_P$. We extend the weight function to $GL(1+n)$ by redefining the weight vector at puncture $q_i$ to be given by 
\begin{equation*}
 \nu_i = (-\epsilon,\nu_i^1,\dotsc,\nu_i^n), 
\end{equation*}
for some $\epsilon > 0$ small, $i= 0, \dotsc , m$. 

\begin{thm}[\cite{pr}, Lemma 4.3 and Lemma 4.5;\cite{Georgios}, Lemma 8.2]\label{thm:fredholm} 
 There is a choice of weight vectors so that the operators $\dbar_{l,A}$, $\dbar_{s,A}$ are Fredholm, and so that for a generic choice of almost complex structures these operators are surjective after having stabilized their domains with the space of conformal variations from Section \ref{sec:confvar}.
\end{thm}

Similar constructions are done for the linearized $\dbar$-operator at a holomorphic disk $u \in \M_L(a, \bf b)$, and we get the same results about Fredholmness and surjectivity. 

\begin{rmk}
Sometimes we write $\dbar_u$ instead of $\dbar_{l,A}$, $\dbar_{s,A}$, to simplify notation. 
\end{rmk}

 All this is related to orientations of moduli spaces in the following way. 
If $u$ is a pseudo-holomorphic disk of $\Lambda$ (or of $L$), it follows from Theorem \ref{thm:fredholm} that $\dbar_u$ is Fredholm. That is, it has finite-dimensional kernel and cokernel. This means that we can consider its determinant line $\det \dbar_u$, 
\begin{equation*}
 \det \dbar_u = \bigwedge^{\max} \Ker \dbar_u \otimes \bigwedge^{\max} (\coker \dbar_u)^*,
\end{equation*}
where $\bigwedge^{\max}V$ is the top exterior power of the vector space $V$.
In particular this means that we can give an orientation to $\det \dbar_u$. This orientation will in turn be related to the orientation of $T_u \M_{i,\Lambda}$ (or  $T_u \M_L$), as we will explain in Section \ref{sec:confvar}.

\section{Orientation conventions}\label{sec:oc}
The signs in the algebraic count of elements in the DGA-morphisms, and also in the DGA-differentials, come from orientations of the moduli spaces of $J$-holomorphic disks. These orientations depend on several choices, which we fix in this section. 

We mainly use the approach of \cite{orientbok} where the moduli spaces $\M_{l,\Lambda}$ are oriented, but we will make some slight modifications of these conventions to make them fit into the symplectization setting.

We close this section by proving that  the chosen conventions imply the statements in Theorem \ref{thm:georgios}.

\subsection{Short exact sequences}
First of all, it is a standard fact that an exact sequence
\begin{equation}\label{eq:es}
 0 \xrightarrow{} V_1 \xrightarrow{\alpha} W_1\xrightarrow{\beta} W_2
\xrightarrow{\gamma} V_2 \xrightarrow{} 0
\end{equation}
of finite-dimensional vector spaces induces an isomorphism  
\begin{equation}\label{eq:phi}
\phi:\bigwedge^{\max}V_1 \otimes  \bigwedge^{\max}V^*_2 \xrightarrow{\approx} 
\bigwedge^{\max}W_1 \otimes  \bigwedge^{\max}W^*_2.
\end{equation} 
 See e.g. [\cite{Floer-Hofer}, Appendix]. 
This isomorphism is not canonical, but depends on choices. For a deeper discussion on this, see \cite{z}.
We will use the following convention, described in terms of oriented bases:

First we identify $\bigwedge^{\max}V^*$ with $\bigwedge^{\max}V$ via 
 $v_1\wedge \dotsb \wedge v_k \mapsto v_1^* \wedge \dotsb \wedge v_k^*$,
where $(v_1,\dotsc,v_k)$ is any basis for $V$, and $v_i^*$ is the vector dual to $v_i$. Now pick a basis $(v_1,\dotsc,v_k)$ for $V_1$, and vectors $(u_1,\dotsc,u_l)\in W_2$ so that $(\gamma(u_1),\dotsc,\gamma(u_l))$ gives a basis for $V_2$. Then pick vectors $(w_1,..,w_m) \in W_1$ so that $(\alpha(v_1),\dotsc,\alpha(v_k), w_1,\dotsc,w_m)$ gives a basis for $W_1$. From the exactness of the sequence \eqref{eq:es} it then follows that $(\beta(w_1),\dotsc,\beta(w_m),u_1,\dotsc,u_l)$ gives a basis for $W_2$. We fix the isomorphism \eqref{eq:phi} to be given by
  \begin{multline}\label{eq:o2}
   v_1\wedge\dotsb \wedge v_k \otimes \gamma(u_1)\wedge\dotsb\wedge \gamma(u_l) \mapsto\\ 
    v_1\wedge\dotsb \wedge v_k \wedge w_1 \wedge \dotsb \wedge w_m \otimes  u_1\wedge\dotsb\wedge u_l \wedge \beta(w_1)\wedge \dotsb \wedge \beta(w_m),
  \end{multline}
and extend by linearity. It is straightforward to check that this definition does not depend on the choice of oriented bases.
 \begin{rmk}\label{rmk:kan}
This convention is slightly different than the one in [\cite{orientbok}, Section 3.2.1].
  As a consequence of this choice, we get rid of the sign $(-1)^{\frac{1}{2}(n-1)(n-2)}$ in the statement of [\cite{orientbok}, Lemma 3.11].
 \end{rmk}

\subsection{Exact gluing sequences, and order of gluing}\label{sec:gluing}
We will repeatedly make use of exact gluing sequences of pseudo-holomorphic disks. For a detailed description we refer to [\cite{orientbok}, Section 3.2]. Here we give an outline of the construction. 

Let  $D_{m_1+1}$ be a disk with punctures $(q_0,q_1,\dotsc,q_{m_1})$ and with an associated Lagrangian boundary condition $A: \partial D_{m_1+1} \to U(n)$. Similarly, let  $D_{m_2+1}$ be a disk with punctures $(p_0,p_1,\dotsc,p_{m_2})$ and with an associated Lagrangian boundary condition $B: \partial D_{m_2+1} \to U(n)$. If $A$ and $B$ are asymptotically equal to the same constant map at the punctures $q_0$ and $p_k$, say, then we can glue $D_{m_1+1}$ to $D_{m_2+1}$ at $q_0 =p_k$, and get a trivialized Lagrangian boundary condition $A\#B$ on the glued disk $D_{m_1+m_2} = D_{m_1+1}\#D_{m_2+1}$.

This gluing induces an exact sequence for the kernels and cokernels of the associated operators $\dbar_A$, $\dbar_B$ and $\dbar_{A\#B}$, given by 
\begin{equation}\label{eq:gluconv}
 0 \to \kd_{A\#B}
 \xrightarrow{\alpha}
 \begin{bmatrix}
  \kd_B \\
  \kd_A
 \end{bmatrix}
 \xrightarrow{\beta}
\begin{bmatrix}
 \cd_B \\
 \cd_A
\end{bmatrix}
 \xrightarrow{\gamma}
\cd_{A\#B} \to 0.
\end{equation}

Here we use the notation 
\begin{equation*}
 \begin{bmatrix}
V\\
W
\end{bmatrix} = V \oplus W.
\end{equation*}

\begin{rmk}\label{rmk:glumap}
 The maps $\alpha$, $\beta$ and $\gamma$ are given as follows. First embed the kernels and cokernels of the $\dbar_A$ and $\dbar_B$-operators in the Sobolev spaces associated to $\dbar_{A \#B}$, by cutting off the elements with cut-off functions $\phi_A^\rho, \phi_B^\rho$. 
Then $\alpha$ is $L^2$-projection
onto the space spanned by the cut-off kernel elements, $\beta$ is
$\dbar$ composed with $L^2$-projection and $\gamma$ is projection
to the quotient. See [\cite{orientbok}, Section 3.2.2] for a more detailed description.
\end{rmk}

Using the isomorphism \eqref{eq:phi} we see that orientations of $\det \dbar_A$ and $\det \dbar_B$ induce an orientation of $\det \dbar_{A\#B}$. Note that this induced orientation depends on the pairwise order of the vector spaces in the second and third column of the gluing sequence \eqref{eq:gluconv}, and that we have chosen the opposite order compared to [\cite{orientbok}, Section 3.2.2]. The reason for this change is that the order in \eqref{eq:gluconv} seems more feasible when working with an extra $\R$-direction, which shows up when we consider pseudo-holomorphic disks in the symplectization instead of in the Lagrangian projection. Compare with the discussion in Subsection \ref{sec:clock}.

\subsection{Orientations of the space of conformal variations}\label{sec:confvar}
Let $u:D_{m+1} \to X$, $X = T^*M$ or $X= \R \times J^1(M)$, be  a rigid $J$-holomorphic disk of $\Lambda$ or of $L$. If $m>1$, then the linearized $\dbar$-operator at $u$, restricted to the Sobolev space of candidate maps, will have cokernel isomorphic to the tangent space of the space of conformal structures of $D_{m+1}$. We call this tangent space the \emph{space of conformal variations}, and the orientation (i.e.\ the sign) of $u$ will depend on which orientation we choose on this space. See [\cite{orientbok}, Lemma 3.17]. We fix this orientation as follows.  

Let $\Co_m$ denote the space of conformal structures on  $D_{m+1}$. If we fix the positions of three of the punctures of $D_{m+1}$, then the position of the other punctures parameterize $\Co_{m}$.  
To describe the orientation of the tangent space $T_{\kappa} \Co_m$ at a conformal structure $\kappa$, let  $\partial_{p_j}$ denote the vector tangent to $\partial D_{m+1}$ at $p_j$, pointing in the counterclockwise direction. Then if we choose $m-2$ of the vectors $\partial_{p_0},\dotsc,\partial_{p_m}$ we get a basis for $T_{\kappa} \Co_m$. We define the positive orientation of $T_{\kappa} \Co_m$ to be given by
\begin{equation}\label{eq:confor}
 (\partial_{p_m}, \dotsc, \partial_{p_3}).
\end{equation}
This somewhat unnatural choice of orientation is a consequence of the convention \eqref{eq:gluconv}. Compare with the discussion in Subsection \ref{sec:clock}.
\begin{rmk}\label{rmk:ob}
 This gives the same orientation as the oriented basis 
 \begin{equation*}
   (\partial_{p_m},\dotsc,\partial_{p_{k+1}},-\partial_{p_{k-1}},\dotsc,-\partial_{p_{j+1}}, \partial_{p_{j-1}},\dotsc,\partial_{p_1}). 
 \end{equation*}
%

\end{rmk}

\begin{rmk}
 If $m \leq 1$ then we can add marked points to the boundary of $D_{m+1}$ to get the setting above. See [\cite{orientbok}, Section 4.2.3].
\end{rmk}

To see how the orientation of the space of conformal variations relates to the sign of a rigid pseudo-holomorphic disk, we consider the \emph{fully linearized} $\dbar$-operator $d\Gamma_u$ at a $J$-holomorphic disk $u$. Here
\begin{equation*}
 d \Gamma_u: \Hi_{2,\nu}[A] \oplus T\Co_{m} \to \Hi_{1,\nu}[0], \qquad d\Gamma_u(v,w) = \dbar_u(v) + \Psi(w), 
\end{equation*}
where $\Psi: T\Co_{m} \to \Hi_{1,\nu}[0]$ is a linear map which we will not specify in detail. The regularity assumptions on the almost complex structure $J$ implies that $d\Gamma_u$ is surjective and that the tangent space of the moduli space $\M$ to which $u$ belongs can be identified with the kernel of $d\Gamma_u$,
\begin{equation*}
 T_u\M \simeq \Ker d\Gamma_u.
\end{equation*}
We see that an orientation of $\Ker d\Gamma_u$ induces an orientation of $T_u\M$, and in particular, if $u$ is rigid so that $\M$ is zero-dimensional this will just be  a sign assigned to $\M$ at $u$. 

By the proof of [\cite{orientbok}, Lemma 3.17] we have that 
\begin{equation*}
 \bigwedge^{\max}\Ker d\Gamma_u \simeq \det \dbar_u \otimes \bigwedge^{\max}T\Co_{m}.
\end{equation*}
Thus an orientation of $\det \dbar_u$ and of $T\Co_{m}$ induces an orientation of $\Ker d\Gamma_u$. In particular, if $u$ is rigid then $\Ker d\Gamma_u = 0 = \kd_u$ (assuming $m>1$) and $\det \dbar_u \otimes \bigwedge^{\max}T\Co_{m}$ is given by the sign of the isomorphism
\begin{equation}\label{eq:psibar}
 \bar{\Psi}:T\Co_{m} \to \cd_u.
\end{equation}
Here $\bar \Psi$ is given by $\Psi$ composed with the projection to the cokernel of $\dbar_u$. In Section \ref{sec:orientdisk} we define an orientation of $\det \dbar_u$ .

Now recall the  gluing of $\dbar_A$ and $\dbar_B$ described in Section \ref{sec:gluing}. The direct sum of the conformal structures of the disks $D_{m_1+1}$ and $D_{m_2+1}$ (which were joined at $q_0 = p_k$, with $q_0$ denoting the positive puncture of $D_{m_1+1}$) can be seen as an element of the boundary of the space $\Co_{m}$, $m=m_1 + m_2-1$. In addition, the outward normal at this conformal structure can be given by $\partial_{q_1}=-\partial_{p_{k-1}}$, or alternatively $\partial_{p_{k+1}} = - \partial_{q_{m_1}}$. We orient the boundary by outward normal last. 

\begin{lma}\label{lma:confvar}
 We have 
 \begin{equation*}
T\Co_{m_2} \oplus T  \Co_{m_1} \oplus \R = (-1)^{(m_1-1)k+1}T\Co_m
 \end{equation*}
as oriented vector spaces, where $\R$ is given the orientation from the outward normal.
\end{lma}
\begin{proof}
 This is similar to the proof of [\cite{orientbok}, Lemma 4.7]. 

\end{proof}

\subsection{Canonical orientation of the closed disk, trivializations, and spin structures}\label{sec:canon}

The determinant line bundle of the $\dbar$-operator over the space of
trivialized Lagrangian boundary conditions on the non-punctured unit disk in
$\C$ is orientable. Moreover,  if we fix an orientation of $\R^n$ and
of $\C$, then
this induces an orientation, via evaluation at the boundary. See
[\cite{fooo}, Proposition 8.1.4]. We denote this induced orientation the \emph{canonical
  orientation} (recall that we assume that we have fixed orientations
of $\C$ and $\R^n$ already, see Remark \ref{rmk:oras}) .

The following proposition follows from [\cite{orientbok}, Section 3.4.2], [\cite{pr}, Section 4.4] and [\cite{fooo}, Section 8.1].
\begin{prp}
 If $\Lambda$ (or $L$) is spin, then a choice of spin structure induces a trivialized Lagrangian boundary condition of $u \in \M_{l,\Lambda}$ (or $u \in \M_L$), which is well-defined up to homotopy.
\end{prp}

%

In this paper we use the following conventions.  If $\Lambda$ is a
spin Legendrian and $L$ is the Lagrangian cylinder $\R \times
\Lambda$, then we give $L$ the spin structure induced  from the spin
structure of $\Lambda$ and the trivial spin structure on $\R$. If $L$
is a spin, exact Lagrangian cobordism with cylindrical ends
$\Lambda_\pm$, then we require that $\Lambda_\pm$ are given the
boundary spin structures induced by $L$.  We refer to
[\cite{orientbok}, Section 4.4] for a  discussion on how other choices of spin structure affect the orientations of the moduli space of pseudo-holomorphic disks. 

From [\cite{orientbok}, Lemma 3.11] together with our orientation convention \eqref{eq:o2} we get the following useful result.

\begin{lma}\label{lma:kan}
 Let $\dbar_A$, $\dbar_B$ be two problems defined on the non-punctured unit disk, where $A$ and $B$ are trivialized Lagrangian boundary conditions. Let $\dbar_{A\#B}$ denote the problem induced by gluing $\dbar_A$ to $\dbar_B$. If $\det \dbar_A$ and $\det \dbar_B$ are given their canonical orientation, and $\R^n$ is given its fixed orientation, then the gluing sequence 
\begin{equation}\label{eq:gluconv2}
 0 \to \kd_{A\#B}
 \xrightarrow{\alpha}
 \begin{bmatrix}
  \kd_B \\
  \kd_A
 \end{bmatrix}
 \xrightarrow{\beta}
\begin{bmatrix}
 \cd_B \\
 \R^n \\
 \cd_A
\end{bmatrix}
 \xrightarrow{\gamma}
\cd_{A\#B} \to 0
\end{equation} 
 induces the canonical orientation on $\det \dbar_{A \# B}$. Here the $\R^n$-summand comes from gluing non-punctured disks, compare [\cite{orientbok}, Lemma 3.1].
\end{lma}

\subsection{Capping operators}\label{sec:orientcap}
Let $u$ be a holomorphic disk of $\Lambda$ or of $L$. As pointed out above, to give a sign to $u$ is related to give an orientation to the determinant line of $\dbar_u$. All this must be done in a coherent way, so that 
we get 
$\partial^2 =0$ and $\Phi_L \circ \partial = \partial \circ \Phi_L$ in the very end. 

The idea from \cite{orientbok} is to use the trivialized Lagrangian boundary condition of $u$, induced by the spin structure of $\Lambda$ or of $L$, together with the canonical orientation of $\det \dbar$ over the space of trivialized Lagrangian boundary conditions on the non-punctured disk. To make this work, we need to choose a way to close up the trivialized boundary conditions of $u$ at the punctures. In \cite{orientbok} this is done by using something called \emph{capping operators}, and this is the method that we will use. We give an outline of the constructions, and also explain the modifications needed to carry it over to the symplectization. 

\subsubsection{Capping trivializations}
The capping operators are $\dbar$-operators defined on the
$1$-punctured unit disk in $\C$, and we have two operators,
$\dbar_{p,+}$ and $\dbar_{p,-}$, associated to each Reeb chord $p$ of
$\Lambda$. The reason for this is that we need one capping operator
for $p$ in the case when $p$ occurs as a positive puncture of a disk,
and another capping operator for $p$ when $p$ occurs as a negative puncture. 

To each capping operator $\dbar_{p,\pm}$ we have an associated trivialized Lagrangian boundary condition $R_{p, \pm}$, which is chosen in a way so that we get a trivialized boundary condition on the non-punctured disk after having glued all the capping operators corresponding to the punctures of $u$ to  $\dbar_u$. We call the boundary conditions $R_{p, \pm}$ the \emph{capping trivializations}.

There are different possibilities to define $R_{p,\pm}$. See for example [\cite{orientbok}, Section 3.3], [\cite{orientbok}, Section 4.5], and [\cite{orienttrees}, Section 3.4.C]. We will not fix a specific system of capping trivializations in the present paper, instead we consider any system that satisfies certain conditions, listed below. In particular, the systems defined in \cite{orientbok} and in \cite{orienttrees} satisfy these conditions. Before stating the conditions, we first have to discuss a stabilization of the tangent bundle of $\Lambda$ and of $L$, made by adding a trivial bundle.

\subsubsection{Auxiliary directions} In [\cite{orientbok}, Section 3.3.3], something called \emph{auxiliary directions} are introduced. These are artificial extra directions that are added to the  capping trivializations and to the Lagrangian trivializations induced by the pseudo-holomorphic disks. The main reason for doing this is to get the invariance proof of Legendrian contact homology over $\Z$ to work out well. These extra directions also simplify the work of assuring that we get a trivialized boundary condition on the non-punctured disk when we glue the capping operators to $\dbar_u$. 

In the case when we are considering Legendrian knots $\Lambda \subset
\R^3$ (i.e.\ when $n=1$) we add one auxiliary direction, to get the stabilized tangent space $\tilde T\Pi_\C(\Lambda) = T \Pi_\C(\Lambda) \oplus \R$. 
In the more general setting when  $n\geq 2$ we add
two auxiliary directions, to get the stabilized tangent space $\tilde T\Pi_\C(\Lambda) = T \Pi_\C(\Lambda) \oplus \R^2$. 
In the case of an exact cobordism $L$ we do the similar thing, so that we get a stabilized tangent space $\tilde TL = T L \oplus \R^i$, where $i = 1$ if $n=1$ and $i=2$ if $n \geq 2$.
We will
in what follows use $d_A$ for the dimension of the auxiliary space
added. That is, if $n=1$ then $d_A=1$, and if $n>1$ then $d_A=2$.

\begin{rmk}
 The reason of adding only one auxiliary direction for $n=1$ is to get compatibility with [\cite{orientbok}, Section 4.5]. We can as well consider the case of adding a $2$-dimensional auxiliary space for all possible $n$. This is easily seen by tracing the proofs in Subsection \ref{sec:pfger} and Section \ref{sec:cob-pf}.  
\end{rmk}

If $u$ is a pseudo-holomorphic disk of $\Lambda$ or of $L$, then the
linearized $\dbar_u$-problem is extended to the auxiliary directions
so that it gives an isomorphism here. See [\cite{orientbok}, Section 3.3.3]. Thus we get a canonical isomorphism between the determinant line of the original $\dbar_u$-problem and the extended one. With abuse of notation, we let $\dbar_u$ denote the extended problem from now on. 

 \subsubsection{System of capping operators for disks in $T^*M$.}
The capping operators $\dbar_{p,\pm}$ are also extended to the
auxiliary directions,  but will in general not give isomorphisms in
these directions. To describe the properties that we require the
capping operators to have, recall that we assume that we have fixed a
trivialization of the Lagrangian boundary conditions of $\dbar_u$ (now
also extended to the auxiliary directions, using the trivial spin
structure here. See [\cite{orientbok}, Section 3.4.2]).
If $p$ is a puncture of $u$, then notice that we have two
Lagrangian subspaces associated to $u$ at $p$, given by the two
stabilized tangent spaces of $\Pi_\C(\Lambda)$ at $p$. From the fixed
trivialization we then get oriented frames for these two spaces. Let
$p_+, p_-$ denote the endpoints of the Reeb chord of $\Lambda$
corresponding to $p$, where $p_+$ corresponds to the end with largest
$z$-coordinate. Let $X_\pm$ denote the oriented frame of the
stabilized tangent space of  $\Pi_\C(\Lambda)$ at $p$  that lifts to $T_{p\pm}\Lambda$.

   We define a \emph{system of capping operators for $\Lambda$} to be
   a set $\mathcal{S}$ consisting of $\dbar$-operators defined on the
   one-punctured unit disk in $\C$, such that for each Reeb chord $p$
   of $\Lambda$ we have a pair of $\dbar$-operators $\dbar_{p,+}$,
   $\dbar_{p,-} \in \Sy$ with associated trivialized boundary
   conditions $R_{p,\pm}: \partial D_1 \to U(n+d_A)$. Moreover, as a
   part of the data of $\Sy$ we choose an orientation of $\det \dbar_{p,-}$ for each Reeb chord $p$. 

   We say that the system is \emph{admissible} if the operators satisfy the following:  
 \begin{enumerate}
   \labitem{{(c1)}}{c1} $R_{p,-}$ takes the oriented frame $X_+$ to the oriented frame $X_-$;
   \labitem{{(c2)}}{c2} $R_{p,+}$ takes the oriented frame $X_-$ to the oriented frame $X_+$;
   \labitem{{(c3)}}{c3}
   \begin{align*}
     &\dim \kd_{p,+} \equiv 0, & \dim \cd_{p,+} &\equiv  |p| + n +d_A+1,& \\
     &\dim \kd_{p,-} \equiv 1, & \dim \cd_{p,-} &\equiv  |p|,&             
   \end{align*}
everything modulo 2.
\end{enumerate}

\begin{rmk}
  The author has not been able to prove Theorem \ref{thm:georgios} -- \ref{thm:differentcappings} for capping
  operators not satisfying \ref{c1} -- \ref{c3}, but believes it
  should be possible.
\end{rmk}

\begin{rmk}\label{rmk:notv}
 Note that  \ref{c1} -- \ref{c3} is not vacuous, since it is satisfied by the capping operators from \cite{orientbok} and from \cite{orienttrees}. This follows from [\cite{orientbok}, Section 3.3.6] for $d_A=2$, [\cite{orientbok}, Section 4.5.2] together with [\cite{legsub}, Proposition 8.14] for $d_A=1$, and from [\cite{orienttrees}, Corollary 3.31].
\end{rmk}

\subsubsection{Capping trivialization in the symplectization-direction}
We extend the boundary conditions $R_{p,\pm}$ to the symplectization,
by defining them to be given by the identity in the
$\R_t$-direction. We denote the induced capping operators by
$\dbar_{s,p,\pm}$, and we use the notation $\dbar_{l,p,\pm}$ for the
capping operators for disks in $T^*M$ defined above (that is,
$\dbar_{l,p,\pm}$ is the restriction of $\dbar_{s,p,\pm}$ to $T^*M$,
but still extended to the auxiliary space).


We need to consider weighted Sobolev spaces to get the capping operators
$\dbar_{s,p,\pm}$ to be Fredholm. To that end, we put a small positive
exponential weight at the puncture in the $\partial_t$-direction. We get the following.

\begin{prp}\label{cor:capiso}
 For each Reeb chord $p$ of $\Lambda$ the projection $\pi_P$ extends to canonical isomorphisms
 \begin{align*}
  \pi_P : \kd_{s,p,\pm} \to \kd_{l,p,\pm},  \\
  \pi_P : \cd_{s,p,\pm} \to \cd_{l,p,\pm}.
 \end{align*}
\end{prp}

\begin{proof}
 We need to prove that in the $\R_t$-direction the capping operators $\dbar_{s,p,\pm}$ are
isomorphisms. But this follows from [\cite{legsub}, Proposition 8.14 and Proposition 8.16] together with the fact that we are considering a $\dbar$-problem with a one-dimensional Lagrangian boundary condition given by $\R$. More directly, this can be seen by using Fourier expansion as in the proof of [\cite{jholo}, Theorem C.4.1].
\end{proof}

\subsubsection{Orientation of capping operators}
Next we define the orientation of the capping operators $\dbar_{s,p,\pm}$, $\dbar_{l,p,\pm}$, by slightly adjusting the constructions from [\cite{orientbok}, Section 3.3] to our situation. 

Recall that for each Reeb chord $p$ we are assumed to  fix an orientation of $\det
\dbar_{l,p,-}$ when we specify our system of capping operators. This will be the \emph{capping orientation  of
  $\dbar_{l,p,-}$}. Notice that this canonically induces an orientation
of $\det \dbar_{s,p,-}$ via the isomorphism in Proposition \ref{cor:capiso}.

To define the orientation of $\det
\dbar_{s,p,+}$, let $\dbar_{s,p}$ denote the $\dbar$-problem on the
non-punctured disk obtained by gluing the $\dbar_{s,p,+}$-problem to the
$\dbar_{s,p,-}$-problem, and consider the induced exact gluing sequence
\begin{equation}\label{eq:glucap}
0 \to
 \Ker \dbar_{s,p} \to 
 \begin{bmatrix}
  \Ker \dbar_{s,p,+} \\
 \R_t\\
 \Ker \dbar_{s,p,-}
 \end{bmatrix}
 \to
 \begin{bmatrix}
  \coker \dbar_{s,p,+}\\
  \coker \dbar_{s,p,-}
 \end{bmatrix}
 \to 
 \coker \dbar_{s,p} \to 0.
\end{equation}
Here the $\R_t$-summand corresponds to a gluing kernel which is born when gluing positive weighted Sobolev spaces, compare [\cite{orienttrees}, Lemma 3.16]. 
The chosen capping orientation of $\dbar_{s,p,-}$ together with the canonical orientation of $\det \dbar_{s,p}$ and the natural orientation of $\R_t$, induces an orientation $\Or(p_+)$  of $\det \dbar_{s,p,+}$, via the sequence \eqref{eq:glucap} and the isomorphism \eqref{eq:phi}.
\begin{defi}\label{def:capor}
  We define the \emph{capping orientation of $\dbar_{s,p,+}$} to be
  $(-1)^{|p|+n + d_A + 1} \Or(p+)$, and we will refer to the $\dbar_{s,p}$ -problem as \emph{the glued capping disk at $p$}.   
\end{defi}

We give the operator $\dbar_{l,p,+}$ the capping orientation induced by  the capping orientation of $\dbar_{s,p,+}$ under the isomorphism $\pi_P$ from Proposition \ref{cor:capiso}. We let $\dbar_{l,p}$ denote the $\dbar$-problem on the non-punctured disk obtained from gluing $\dbar_{l,p,+}$ to $\dbar_{l,p,-}$, and we prove the following.
 
\begin{lma}\label{lma:lma1}
  Assume that we have chosen an admissible system of capping operators
  for $\Lambda$.
  Then the capping orientations of $\dbar_{l,p,+}$ and $\dbar_{l,p,-}$ glue to the canonical orientation of $\dbar_{l,p}$, times $(-1)^{|p|+n+d_A+1}$, under the exact gluing sequence
\begin{equation}\label{eq:glucaplag}
0 \to \Ker \dbar_{l,p} \to 
 \begin{bmatrix}
  \Ker \dbar_{l,p,+} \\
 \Ker \dbar_{l,p,-}
 \end{bmatrix}
 \to
 \begin{bmatrix}
  \coker \dbar_{l,p,+}\\
  \coker \dbar_{l,p,-}
 \end{bmatrix}
 \to 
 \coker \dbar_{l,p} \to 0.
\end{equation}

 \end{lma}

 \begin{rmk}
 Since we use the gluing convention \eqref{eq:gluconv}, this gives the opposite convention of [\cite{orientbok}, Section 3.3.4]. Also notice that the sign $(-1)^{|p|+n+d_A+1}$ in the definition of the orientation of the capping operators is not used in that paper. 
 \end{rmk}

 \begin{proof}[Proof of Lemma \ref{lma:lma1}]
By construction we have that $\kd_{s,p} \simeq \R_t \oplus \kd_{l,p}$, and     
 from Remark \ref{rmk:glumap} we see that the first nontrivial map in  \eqref{eq:glucap} restricted to the $\R_t$-factors is given
by projection $v \mapsto v$. Indeed, the kernel that is born during the gluing is cut-off and embedded in the Sobolev space of the glued map, and the gluing map $\alpha$ in \eqref{eq:gluconv} was given by $L^2$-projection. Since by the assumptions \ref{c3} we have $\dim \kd_{s,p,+} \equiv 0$ modulo 2, we can remove $\R_t$ from both
the first and second nontrivial column without affecting orientations on the remaining spaces. But after removing $\R_t$ we get the gluing sequence for the capping operators in the Lagrangian projection, and since the canonical orientation is given via evaluation the result follows.  
 \end{proof}

From now on, we use the notation $\dbar_{p,\pm}$ to denote the capping operators both in the symplectization-setting and in the setting of the Lagrangian projection.

\subsection{Capping orientation of disks}\label{sec:orientdisk}
Now we give the definition of the capping orientation of a punctured pseudo-holomorphic disk. Below $X$ denotes either $T^*M$ or $\R\times J^1(M)$, with almost complex structure $J$ as described in Section \ref{sec:holo}.

If $L$ is an exact Lagrangian cobordism with cylindrical Legendrian ends $\Lambda_\pm$, then assume that $\Sy_\pm$ gives a system of capping operators for $\Lambda_\pm$. This  gives rise to an \emph{induced system of capping operators of $L$}, where the positive capping operators $\dbar_{p,+}$ are taken from the system $\Sy_+$, and the negative capping operators $\dbar_{p,-}$ are taken from $\Sy_-$. This system is \emph{admissible} if both $\Sy_\pm$ are admissible. 

 Let $u :D_m \to X$ be a pseudo-holomorphic disk of $\Lambda$ or of $L$, with positive puncture $a$ and negative punctures $b_1,\dotsc,b_m$. Assume that we have fixed a system of capping operators,  and consider the exact gluing sequence 
 \begin{equation}\label{eq:caporient}
  0 \to \Ker \dbar_{\hat u} 
  \to
  \begin{bmatrix}
   \Ker \dbar_u \\
   \Ker \dbar_{a,+} \\
   \Ker \dbar_{b_m,-} \\
   \vdots\\
   \Ker \dbar_{b_1,-}
  \end{bmatrix}
\to  
\begin{bmatrix}
  \coker \dbar_u \\
  \coker \dbar_{a,+} \\
  \coker \dbar_{b_m,-} \\
  \vdots\\
  \coker \dbar_{b_1,-}
  \end{bmatrix}
  \to
  \coker \dbar_{\hat u}
  \to 0.
 \end{equation}
Here $\dbar_{\hat u}$ denotes the $\dbar$-problem on the non-punctured
disk with trivialized boundary condition $\hat u$, which is obtained
by gluing the trivialized boundary condition induced by $u$ to the
positive capping trivialization of $a$ at the positive puncture of
$D_m$, and then to the negative capping trivializations of
$b_m,\dotsc,b_1$ at the corresponding negative punctures. We refer to
the $\dbar_{\hat u}$-problem as the \emph{fully capped problem
  corresponding to $u$}, and to the sequence \eqref{eq:caporient} as the
\emph{capping sequence for $u$}.

\begin{rmk}
 The reason that the capping operators are glued clockwise along $u$, instead of counterclockwise as in [\cite{orientbok}, Section 3.3.5], is because of our choice of convention \eqref{eq:gluconv}, which in turn depended on the fact that we have to take into account the extra $\R$-direction coming from the symplectization. This is explained further in Section \ref{sec:clock}.
\end{rmk}

Let $\ind (u)$ denote the Fredholm index of $\dbar_u$, restricted to the space of candidate maps. That is, if $u$ is rigid and has $m$ negative punctures, then
\begin{equation*}
 |\ind (u)| \equiv m \pmod{2}.
\end{equation*}

\begin{defi}\label{def:capping}
  We define the \emph{capping orientation} of $u \in \M_{l, \Lambda}(a, {\bf b}),\M_L(a, {\bf b})$ to be the orientation $\Or(\dbar_u)$ on $\det \dbar_u$ induced by the gluing sequence \eqref{eq:caporient}, where $\det \dbar_{\hat u}$ is given the canonical orientation, and the capping operators are given their capping orientations. For $ u \in \M_{s, \Lambda}(a, {\bf b})$ we define the capping orientation to be given by $\det \dbar_{\pi_p(u)}\wedge \partial_{t} = (-1)^{\ind (\pi_p(u))}\partial_{t} \wedge \det \dbar_{\pi_p(u)}$, where $\det \dbar_{\pi_p(u)}$ is given its capping orientation and $\partial_{t}$ gives the positive orientation in the symplectization direction. 
\end{defi}

\begin{rmk}
 We  will use the notation $\partial_{\Sy}$ to indicate the dependence of the chosen capping system $\Sy$ in the definition for the DGA-differential.
\end{rmk}

In the case when $u$ is a rigid disk (and where we assume that we have
divided out the $\R_t$-action if $u$ is a disk of $\R \times
\Lambda$), the capping orientation of $u$ can be understood as an
orientation of the kernel or the cokernel of $\dbar_u$. Moreover, by
[\cite{orientbok}, Section 4.2.3], we may assume that we are in the
case when $\kd_u$ is trivial, so that an orientation of $\det \dbar_u$
is nothing but an orientation of $\cd_u$. Let $\kappa$ denote the
conformal structure of $u$, and recall that the space of
conformal variations at $\kappa$ was given a fixed
orientation in Section \ref{sec:confvar}. Thus, we can compare the capping
orientation of $u$ with 
this orientation via the isomorphism \eqref{eq:psibar} and  get a sign $\sigma(u) \in \{-1,1\}$. This sign is
the \emph{capping sign of $u$}, and is the one that we use in the the
algebraic count of the elements in the moduli spaces when defining the
DGA-morphisms and the DGA-differentials in (\ref{eq:differential}) and
(\ref{eq:phidisk}), respectively. See [\cite{orientbok},
Section 3.4.3].

\begin{rmk}\label{rmk:onedim}
In the case when $u$ belongs to a one-dimensional moduli space $\M$ (after having divided out the $\R_t$-action if $u$ is a disk of $\R\times \Lambda$) we get that the map given in \eqref{eq:psibar} is not an isomorphism, but is surjective with one-dimensional kernel $\Ker \bar \Psi$. Give this space the orientation so that
\begin{equation*}
 T\Co_{m} = \Ker \bar \Psi \oplus \cd_u,
\end{equation*}
as oriented vector spaces. Here $\Co_{m}$ is the space of conformal variations at $u$ and is given its fixed orientation, and $\cd_u$ is given its capping orientation.
The \emph{capping orientation of $\M$ at $u$} is given by this orientation of $\Ker \bar \Psi$, after having identified $\Ker \bar \Psi$ with $\Ker d\Gamma_u$.
\end{rmk}

\subsection{Proof of Theorem \ref{thm:georgios}}\label{sec:pfger}
To prove Theorem \ref{thm:georgios} it only remains to establish the following.

\begin{lma}\label{lma:lma3}
 Let $\Sy$ be an admissible system of capping operators and let $i=s$ or $i=l$. Then the map defined by
   \begin{equation*}
  \partial_{\Sy,i} a = \sum_{\dim \M(a,{\bf b}) = 0} | \M_{i,\Lambda}(a, {\bf b})| {\bf b}
 \end{equation*}
on generators and extended by the signed Leibniz rule to rest of the algebra, satisfies $\partial_{\Sy,i}^2 =0$. Here $| \M_{i,\Lambda}(a, {\bf b})|$ is the algebraic count of disks in $ \M_{i,\Lambda}(a, {\bf b})$, where each disk is counted with its capping sign induced by $\Sy$.
\end{lma}
\begin{proof}
We follow the proof of Theorem 4.1 in \cite{orientbok}, and in
particular the notations therein. Briefly, the argument goes as
follows.

Let $a$ be a Reeb chord of $\Lambda$, and assume that $\partial^2(a)$ has a summand $N{\bf d}$, where $N \in \Z$ and $ {\bf d}= d_1 \cdots d_l$ is a word of Reeb chords. Then $\M(a, {\bf d})$ is one-dimensional and can be compactified by broken pairs of disks from zero-dimensional moduli spaces. Let $\M \subset \M(a, {\bf d})$ be a component with oriented boundary 
\begin{equation}
 \partial \M = \M_1 - \M_0,
\end{equation} 
where $\M_0$ is given by 2 broken disks $(u_1, u_2)$ with
\begin{align*}
 & u_2 \in \M(a; d_1 \cdots d_{k-1} c d_{k+m+1} \cdots d_l),\\
 & u_1 \in \M(c; d_k \cdots d_{k+m}),
\end{align*}
and $\M_1$ is given by 2 broken disks $(u_1', u_2')$ with
\begin{align*}
 & u_2' \in \M(a; d_1 \cdots d_{k'-1} c' d_{k'+m'+1} \cdots d_l),\\
 & u_1' \in \M(c'; d_{k'} \cdots d_{k'+m'}).
\end{align*}

Let $\mu_i$ be the capping sign of $u_i$, and let $\mu_i'$ be the capping sign of $u_i'$, $i= 1,2$. Then the broken disk $(u_1, u_2)$ contributes with $(-1)^{\sum_{i=1}^{k-1}|d_i|} \mu_1 \mu_2 {\bf d}$ to $\partial^2(a)$, and the broken disk $(u_1', u_2')$ contributes with $(-1)^{\sum_{i=1}^{k'-1}|d_i|} \mu_1' \mu_2' {\bf d}$. We claim that 
\begin{equation}\label{eq:claim}
 (-1)^{\sum_{i=1}^{k-1}|d_i|} \mu_1 \mu_2 = - (-1)^{\sum_{i=1}^{k'-1}|d_i|} \mu_1' \mu_2'.
\end{equation} 

To prove this claim we would like to use the orientation of $\M$. That is, by gluing, we have that $u_1\# u_2$  and $u_1'\# u_2'$ give elements in $\M$ with associated $\dbar$-operators $\dbar_{u_1\#u_2}$, $\dbar_{u_1'\#u_2'}$, equipped with capping orientations. Let $o_c(\dbar_{u_1 \# u_2})$ denote the capping orientation of $\det \dbar_{u_1 \# u_2}$ and let $o_i(\dbar_{u_1 \# u_2})$ denote the orientation induced on $\det \dbar_{u_1 \# u_2}$ by the gluing sequence 
\begin{equation}\label{eq:ny01}
0 \to
   \kd_{u_1 \# u_2}
\to 
\begin{bmatrix}
  \kd_{u_2} \\
  \kd_{u_1} \\
  \end{bmatrix}
\to
\begin{bmatrix}
  \cd_{u_2} \\
 \cd_{u_1} \\
\end{bmatrix}
\to
\cd_{u_1 \# u_2}
\to 0
\end{equation} 
in the setting of the Lagrangian projection, and by the sequence 
\begin{equation}\label{eq:nys01}
0 \to
   \kd_{u_1 \# u_2}
\to 
\begin{bmatrix}
  \kd_{u_2} \\
  \kd_{u_1} \\
  \end{bmatrix}
\to
\begin{bmatrix}
  \cd_{u_2} \\
  \R_t \\
 \cd_{u_1} \\
\end{bmatrix}
\to
\cd_{u_1 \# u_2}
\to 0
\end{equation} 
in the symplectization setting. Here we assume that $\det \dbar_{u_i}$ is given its capping orientation for $i=1,2$.  
The $\R_t$-summand in \eqref{eq:nys01} comes from a gluing cokernel that is born when we glue negative weighted problems, see [\cite{orientbok}, Lemma 3.1]. We define $o_c(\dbar_{u'_1 \# u'_2})$, $o_i(\dbar_{u'_1 \# u'_2})$ in a completely analogous way. 
  
Let 
\begin{equation*}
\delta = \dim \cd_{u_1} + \dim \cd_{u_2} +1 = \dim \cd_{u_1'} + \dim \cd_{u_2'} +1 
\end{equation*}
 if we are in the symplectization setting, and 
 \begin{equation*}
\delta = n + d_A+1  
 \end{equation*}
if we are in the setting of the Lagrangian projection. The claim \eqref{eq:claim} will follow if we can prove that
\begin{align}\label{eq:holm1}
 & o_c(\dbar_{u_1 \# u_2}) = (-1)^{(\dim \cd_{u_1}-1)k + \sum_{i=1}^{k-1}|d_i|+ \delta} o_i(\dbar_{u_1 \# u_2})  \\ \label{eq:holm2}
 & o_c(\dbar_{u'_1 \# u'_2}) = (-1)^{(\dim \cd_{u'_1}-1)k' + \sum_{i=1}^{k'-1}|d_i|+\delta} o_i(\dbar_{u'_1 \# u'_2}).
\end{align}
Before we derive these equations, we prove that they 
 imply \eqref{eq:claim}. Indeed, to relate the orientation of $\M$ with \eqref{eq:holm1}, we should consider the following commutative diagram
\begin{equation}\label{eq:twostar}
\xymatrix{
 T\Co_{u_2}\oplus T\Co_{u_1} \oplus \R \ar[d] \ar[r] & T\Co_{u_1\#u_2} \ar[d] \\
 \cd_{u_2} \oplus \cd_{u_1} \ar[r] & \cd_{u_1\#u_2}.
 }
\end{equation}
Here $\Co_v$ denotes the space of conformal structures of the domain of the  disk $v$, $v=u_1, u_2, u_1\#u_2$. This commuting diagram is similar to the one given in the proof of [\cite{orientbok}, Lemma 4.11], and the maps in the diagram are given as follows. 

\begin{enumerate}
 \labitem{{(hu)}}{hu} The upper horizontal map is given by the map in Lemma \ref{lma:confvar}.
 \labitem{{(hl)}}{hl} The lower horizontal map is the one induced by the gluing sequence \eqref{eq:ny01} in the setting of the Lagrangian projection and induced by the gluing sequence \eqref{eq:nys01} in the symplectization setting.
 \labitem{{(vl)}}{vl} The left vertical map is given by $(x,y,t) \mapsto(\bar\Psi_1(x), \bar \Psi_2(y))$, where $\bar \Psi_i$ is the map \eqref{eq:psibar} associated to $u_i$, $i=1,2$. 
 \labitem{{(vr)}}{vr} The right vertical map is given by $z \mapsto \bar \Psi(z)$, where $\bar \Psi$ is the map \eqref{eq:psibar} associated to $u_1\#u_2$. 
\end{enumerate}

Assume that $T\Co_{u_1\#u_2}, T\Co_{u_1}, T\Co_{u_2}$ are given their fixed orientations and that $\R$ is endowed with the orientation corresponding to the outward normal of $T\Co_{u_1\#u_2}$, as in Section \ref{sec:confvar}. Also assume that $\cd_{u_1\#u_2}, \cd_{u_1}, \cd_{u_2}$ are given their capping orientations. Then Lemma \ref{lma:confvar} implies that \ref{hu} is an isomorphism between oriented spaces of sign $(-1)^{\nu_1}$, 
\begin{equation*}
 \nu_1 = (\dim T\Co_{u_1}-1)\cdot k +1, 
\end{equation*}
and \ref{vl} restricted to $T\Co_{u_1} \oplus T\Co_{u_2}$ is by definition an isomorphism between oriented spaces of sign $(-1)^{\nu_2}$,
\begin{equation*}
 (-1)^{\nu_2} = \mu_1\mu_2.
\end{equation*}
From \eqref{eq:holm1} we get that the sign of \ref{hl} is given by $(-1)^{\nu_3}$,
\begin{equation*}
 \nu_3 = (\dim \cd_{u_1}-1)\cdot k + \sum_{i=1}^{k-1}|d_i|+ \delta.
\end{equation*}
Thus, by Remark \ref{rmk:onedim} together with the commutativity of \eqref{eq:twostar} we see that the orientation of $T_{u_1\#u_2}\M$ is given by 
\begin{align*}
 (-1)^{\nu_1 + \nu_2 + \nu_3}\R &= (-1)^{(\dim T\Co_{u_1}-1)\cdot k +1 + (\dim \cd_{u_1}-1)\cdot k + \sum_{i=1}^{k-1}|d_i|+ \delta}\mu_1\mu_2 \R\\
 &= (-1)^{1+\sum_{i=1}^{k-1}|d_i|+ \delta}\mu_1\mu_2\R,
\end{align*}
where $\R$ is given the orientation of the outward normal of $T\Co_{u_1\#u_2}$ at $u_1\#u_2$. 

Using the same argument at $u'_1\#u'_2$ we get in total that the orientation of $\M$ at the boundary component  $\M_0$ is given by $(-1)^{1+\sum_{i=1}^{k-1}|d_i|+ \delta}\mu_1\mu_2$ times the outer normal and the orientation of $\M$ at the boundary component $\M_1$ is given by $(-1)^{1+\sum_{i=1}^{k'-1}|d_i|+ \delta}\mu'_1\mu'_2$ times the outer normal. But since the orientation of $\M_1$ is opposite to the orientation of $\M_0$ we get \eqref{eq:claim}.

To prove \eqref{eq:holm1} and \eqref{eq:holm2} , consider the general situation where     
$A$ and $B$ are trivialized Lagrangian boundary conditions
 associated to punctured disks $D_{m}$ and $D_{r+1}$,
 respectively, and assume that we glue them together as described in
 Section \ref{sec:gluing}. That is, we glue the positive puncture of
 $\dbar_A$ to the $k$:th negative puncture of $\dbar_B$. 

Note that $\kd_{s,A} \simeq \R_t \oplus \kd_{l,A}$, $\cd_{s,A} \simeq  \cd_{l,A}$, and similarly for $\hat A$, $B$ and $\hat B$. We write $\dbar_{C}=\dbar_{l,C} $ to simplify notation, where $C$ is any trivialized Lagrangian boundary condition on the (possibly punctured) disk. 

Assume that $\dbar_A$ corresponds to a disk $u_1 \in \M(b_k,f_1\cdots f_{m-1})$ and that $\dbar_B$ corresponds to a disk $u_2 \in \M(a,b_1\cdots b_r)$, and that both these disks are rigid. Also assume that we are in the symplectization setting (the computations in the setting of the
Lagrangian projection are similar), and that $m>2, r>1$.

The sequence \eqref{eq:nys01} now reads
\begin{equation}\label{eq:ny1}
0 \to
   \kd_{s,A \# B}
\to 
\begin{bmatrix}
  \kd_{s,B} \\
  \kd_{s,A} \\
  \end{bmatrix}
\to
\begin{bmatrix}
  \cd_{s,B} \\
  \R_t \\
 \cd_{s,A} \\
\end{bmatrix}
\to
\cd_{s, A \#  B}
\to 0,
\end{equation}
which moreover can be simplified to 
\begin{equation}\label{eq:ny001}
0 \to
   \R_t^{A\#B}
\to 
\begin{bmatrix}
  \R_t^B \\
  \R_t^A 
  \end{bmatrix}
\to
\begin{bmatrix}
  \cd_{B} \\
  \R_t \\
 \cd_{A} 
\end{bmatrix}
\to
\cd_{A \#  B}
\to 0.
\end{equation}
Here we use the notation $\R_t^{A\#B}$ for the $\R_t$-factor in $\kd_{A\#B}$, and similar for $\R_t^A$, $\R_t^B$. We have also used the assumption that $m>2, r>1$, so that $\kd_A=\kd_B=\kd_{A\#B}=0$. 

By Remark \ref{rmk:glumap} and the proof of [\cite{orientbok}, Lemma 3.11], we may assume that the first non-trivial map in \eqref{eq:ny001} is given by $t \mapsto (t,t)$ and that the second map is given by $(s,t) \mapsto (0,t-s,0)$. By our orientation conventions \eqref{eq:o2} it follows that if we move $\R_t$ over $\cd_A$, with a cost of $(-1)^{\sigma_0}$,
\begin{equation}\label{eq:cokersign}
 \sigma_0 = \dim \cd_A,
\end{equation}
we can reduce \eqref{eq:ny001} to 
\begin{equation}
 0 \to
   0
\to 
0
\to
\begin{bmatrix}
  \cd_{B} \\
 \cd_{A} 
\end{bmatrix}
\to
\cd_{A \#  B}
\to 0,
\end{equation}
where the non-trivial map is the one that occurs in \eqref{eq:twostar} as \ref{hl}.

%
%

To compute the difference between the capping orientation of $\det \dbar_{s,A\#B}$ and the orientation induced by the capping orientations of $\det \dbar_{s,A}$ and $\det \dbar_{s,B}$ from the sequence \eqref{eq:ny001}, consider the following two gluing sequences 
\begin{equation}\label{eq:4.9.1}
  \begin{bmatrix}
   \R_t \\
   \kd_{\hat A \# \hat B}
  \end{bmatrix}
\to 
\begin{bmatrix}
 \begin{pmatrix}
  \R_t^B \\
  \hbp_{j=1}^r \kd_{b_{j},-} \\
 \end{pmatrix}\\
 \begin{pmatrix}
  \R_t^A \\
  \hbp_{j=1}^{m-1} \kd_{f_{j},-}
 \end{pmatrix}
\end{bmatrix}
\to
\begin{bmatrix}
 \begin{pmatrix}
  \cd_B \\
  \cd_{a,+} \\
  \hbp_{j=1}^r \cd_{b_{j},-}
 \end{pmatrix}\\
\R_t \\
\R^{n+d_A}\\
\begin{pmatrix}
 \cd_A \\
 \cd_{b_{k},+} \\
 \hbp_{j=1}^{m-1} \cd_{f_{j},-}
\end{pmatrix}
\end{bmatrix}
\to
\cd_{\hat A \# \hat B},
\end{equation}

\begin{equation}\label{eq:4.9.2}
  \begin{bmatrix}
   \R_t \\
   \kd_{\hat A \# \hat B}
  \end{bmatrix}
\to 
\begin{bmatrix}
 \begin{pmatrix}
   \R_t^B \\
   \R_t^A\\
  \hbp_{j=k+1}^r \kd_{b_{j},-} \\
  \hbp_{j=1}^{m-1} \kd_{f_{j},-} \\
  \hbp_{j=1}^{k-1} \kd_{b_{j},-} 
 \end{pmatrix}\\
 \begin{pmatrix}
  \R_t \\ 
  \kd_{b_{k},-}
 \end{pmatrix}
\end{bmatrix}
\to
\begin{bmatrix}
 \begin{pmatrix}
   \cd_B \\
   \R_t\\
  \cd_A \\
  \cd_{a,+} \\
  \hbp_{j=k+1}^r \cd_{b_{j},-}\\
 \hbp_{j=1}^{m-1} \cd_{f_{j},-} \\
  \hbp_{j=1}^{k-1} \cd_{b_{j},-}\\
 \end{pmatrix}\\
\R_t \\
\R^{n+d_A}\\
\begin{pmatrix}
 \cd_{b_{k},+} \\
 \cd_{b_{k},-}
\end{pmatrix}
\end{bmatrix}
\to
\cd_{\hat A \# \hat B}.
\end{equation}
Here we use the notation $\hbp_{j=1}^r V_j = V_r \oplus \dotsb \oplus V_1$. We have also omitted the trivial maps $0 \to \R_t \oplus \kd_{\hat A \# \hat B}$ and $\cd_{\hat A \# \hat B} \to 0$ to fit the sequences within page margins. In addition, we have dropped some spaces which by assumption have even dimension, since they will not affect the calculations below.  

The first sequence corresponds to the gluing of $\dbar_{s,\hat A}$ to $\dbar_{s,\hat B}$, and the second one to the gluing of $\dbar_{s,\widehat{A\#B}}$ to the glued capping disk at $b_k$. Here we have used the associativity of orientations under gluing, see [\cite{orientbok}, Section 3.2.3], together with the capping sequences for $\dbar_{s,A}$, $\dbar_{s,B}$ and $\dbar_{s,A\#B}$, respectively. By Lemma \ref{lma:kan} we get that both sequences induce the canonical orientation on $\dbar_{s,\hat A \# \hat B}$, given that   
$\dbar_{s,\hat A}$, $\dbar_{s,\hat B}$, $\dbar_{s,\widehat{A\#B}}$ and the glued capping disk at $b_k$ all are given their canonical orientation. 

\begin{rmk}\label{rmk:forenkla}
Notice that, if we compensate by an overall sign \linebreak
$(-1)^{\ind \dbar_A + \ind \dbar_B + \ind \dbar_{A\#B}} = 1$ at the end, we can instead of the capping orientations of $\dbar_{s,A}, \dbar_{s,B}, \dbar_{s,A\#B}$ consider the capping orientations of $\dbar_{A}, \dbar_{B}, \dbar_{A\#B}$ (assuming that all $\R_t$-summands are given their natural orientation). This follows from Definition \ref{def:capping}.
\end{rmk}

Thus, if we assume that $\dbar_A$ and $\dbar_B$ are given their capping orientations, we can compute the difference between the capping orientation of $\dbar_{A\#B}$ and the one induced by the sequence \eqref{eq:ny001} by rearranging the spaces in  \eqref{eq:4.9.1} to be in the same order as the spaces in \eqref{eq:4.9.2}, in the same time as we keep track of the change in orientations that these rearrangements induce. 

%
%
%
%

First, by Remark \ref{rmk:glumap} and the proof of [\cite{orientbok}, Lemma 3.11], we may assume that the bottom-most $\R_t$-summand in
the second column of \eqref{eq:4.9.2} is mapped by the identity to the
bottom-most $\R_t$-summand in the third column. Thus these two spaces
can be removed if we compensate by a sign
\begin{equation*}
  (-1)^{\dim \kd_{b_k,-} + \dim \R^{n+d_A} + \dim \cd_{b_k,+} + \dim \cd_{b_k,-}}=(-1)^{1 + n + d_A + |b_k| +n+1+d_A+|b_k|} = 1.
\end{equation*}
Here we have used \ref{c3} repeatedly.

Next, notice that by \ref{c3}, \eqref{eq:dim1} and by the assumption that $u_2$ is rigid,  we have 
\begin{equation}\label{eq:forenk}
  \dim \cd_{a,+} +  \dim \left(\hbp_{j=1}^{r}\cd_{b_{j},-}\right) \equiv  |a|+1 +n+d_A + \sum_{j=1}^{r}|b_j| \equiv n+d_A \pmod{2}.
\end{equation}
Thus, by moving $\R_t^A$ in the sequence \eqref{eq:4.9.1} to the
position right under $\R_t^B$, in the same time as we move the
$\R_t$-factor in the third column to the position right under $\cd_B$, we get a sign 
$(-1)^{\sigma_1}$,
\begin{align*}
  \sigma_1 &= 1\cdot \dim\left(\hbp_{j=1}^{r}\kd_{b_{j},-}\right) + 1 \cdot\left( \dim\left(\hbp_{j=1}^{r}\cd_{b_{j},-}\right) + \dim  \cd_{a,+} \right) \\  
  &\equiv r + n +d_A \pmod{2}.
\end{align*}
 Again we have used \ref{c3}. 

Then we move $\cd_A$, which has dimension $m-1$ modulo 2, to the place just below $\cd_B \oplus \R_t$. Using \eqref{eq:forenk} we see that this costs 
\begin{equation*}
 (-1)^{(m-1)(n+d_A +n+d_A )} =1.
\end{equation*}
Then we change places of $ \hbp_{j=1}^{m-1} \kd_{f_{j},-}$ and $\kd_{b_{k},-}$, which costs $(-1)^{\sigma_2}$,
\begin{align*}
 \sigma_2 &= \left(\dim \left(\hbp_{j=1}^{m-1} \kd_{f_{j},-}\right) + \dim \kd_{b_{k},-}\right) \cdot \dim \left(\hbp_{j=1}^{k-1} \kd_{b_{j},-}\right)\\ 
 &\quad +  \dim \left(\hbp_{j=1}^{m-1} \kd_{f_{j},-}\right) \cdot \dim \kd_{b_{k},-}\\
 &\equiv m(k-1) + m-1 \equiv mk +1 \pmod{2},
\end{align*}
and then we do the similar thing for the cokernels. Using that 
\begin{equation*}
 \dim \left(\hbp_{j=1}^{m-1} \cd_{f_{j},-}\right) \equiv  \sum_{j=1}^{m-1}|f_j| \equiv |b_k| +1 \pmod{2},
\end{equation*}
again by \ref{c3}, \eqref{eq:dim1} and since $u_1$ is rigid, we see that this permutation  costs $(-1)^{\sigma_3}$,
\begin{align*}
 \sigma_3 &= \left(\dim \left(\hbp_{j=1}^{m-1} \cd_{f_{j},-}\right) + \dim \cd_{b_{k},-}\right) \cdot \\ &\quad \left(\dim \cd_{b_{k},+} + \dim \R^{n+d_A}   + \dim \left(\hbp_{j=1}^{k-1} \cd_{b_{j},-}\right)\right)\\ 
 &\quad +  \dim \left(\hbp_{j=1}^{m-1} \cd_{f_{j},-}\right) \cdot \dim \cd_{b_{k},-}\\
 &\equiv(|b_k| +1 + |b_k|)(|b_k| +1 + n +d_A  + n + d_A + \sum_{i=1}^{k-1}|b_i|) + (|b_k| +1)|b_k|\\
 &\equiv \sum_{i=1}^{k}|b_i| +1 \pmod{2}.
\end{align*}
%

Thus we get a total sign $(-1)^{\sigma}$,
\begin{equation*}
 \sigma = \sigma_1 + \sigma_2 + \sigma_3 \equiv r +n +d_A+ mk + |b_k| + \sum_{i=1}^{k-1}|b_i|  \pmod{2}. 
\end{equation*}



Now, notice that the second gluing sequence does not induce the
canonical orientation of $\det \dbar_{\hat A \# \hat B}$, but
$(-1)^{|b_k|+ n+d_A+1}$ times the canonical orientation. This is because of Definition \ref{def:capor}, which implies that 
\begin{equation*}
0 \to
 \Ker \dbar_{s,b_{k}} \to 
 \begin{bmatrix}
  \Ker \dbar_{s,b_{k},+} \\
 \R_t\\
 \Ker \dbar_{s,b_{k},-}
 \end{bmatrix}
 \to
 \begin{bmatrix}
  \coker \dbar_{s,b_{k},+}\\
  \coker \dbar_{s,b_{k},-}
 \end{bmatrix}
 \to 
 \coker \dbar_{s,b_{k}} \to 0
\end{equation*}
induces the canonical orientation times $(-1)^{|b_k|+n+d_A+1}$ on $\det\dbar_{s,b_{k}}$, given that the capping operators are given their capping orientations. (In the Lagrangian setting, the similar conclusion follows from Lemma \ref{lma:lma1}.)

Hence the difference in orientation of $\det \dbar_{\hat A \# \hat B}$ given by \eqref{eq:4.9.1} and \eqref{eq:4.9.2}, assuming the capping operators are given their capping orientations, is given by $(-1)^{\tilde \sigma}$,
\begin{equation}\label{eq:sis}
 \tilde \sigma = \sigma + |b_k|+n+d_A+1 \equiv mk + \sum_{i=1}^{k-1}|b_i| +r +1 \pmod{2}.
\end{equation}

Now recall form Definition \ref{def:capping} and Remark \ref{rmk:forenkla} that we would like to express the capping orientation of $\det \dbar_{s,A\#B}$ as a wedge $\partial_t \wedge \det \dbar_{A\#B}=\partial_t \wedge \bigwedge^{\max}\cd_{A\#B}$. From \eqref{eq:cokersign} and \eqref{eq:sis} it follows that 
\begin{equation*}
 \partial_t \wedge \bigwedge^{\max}\cd_{A\#B} = (-1)^{\tilde \sigma + \dim \cd_A}\partial_t  \wedge \bigwedge^{\max}\cd_{B} \wedge \bigwedge^{\max}\cd_{A}
\end{equation*}
given that $\cd_A, \cd_B$ and $\cd_{A\#B}$ are given their capping orientations. Using Remark \ref{rmk:forenkla} again, noting that 
\begin{equation*}
 \tilde \sigma + \dim \cd_A \equiv (\dim \cd_{A}-1)\cdot k + \sum_{i=1}^{k-1}|b_i| + \dim \cd_{A} + \dim \cd_{B} +1
\end{equation*}
we deduce that \eqref{eq:holm1} holds. 

\end{proof}

\begin{proof}[Proof of Theorem \ref{thm:georgios}]
 The first statement follows by [\cite{Georgios}, Theorem 2.1]
 together with Definition \ref{def:capping}. The second statement follows from Lemma \ref{lma:lma3}.
\end{proof}

\section{DGA-morphisms induced by exact Lagrangian cobordisms}\label{sec:cob-pf}
Let $L$ be an exact Lagrangian cobordism with cylindrical Legendrian ends
$\Lambda_\pm$, and assume that we are given admissible systems $\Sy_\pm$ of
capping operators for $\Lambda_\pm$. Let $\partial_\pm =
\partial_{\Sy_\pm}$ to simplify notation.
Moreover, let $\Sy$ denote the induced system of $L$, and let $\Phi_L:
(\A(\Lambda_+), \partial_+, \Z) \to (\A(\Lambda_-), \partial_-, \Z)$
be defined by \eqref{eq:phidisk} -- \eqref{eq:ex2}, where now $|\M_L(a,{\bf b})|$ denotes
the algebraic count induced by the capping orientation of $\M_L$
corresponding to the system $\Sy$. When we want to emphasize that
$\Phi_L$ is given with respect to the system $\Sy$ we write
$\Phi_{L,\Sy}$. 

In this section we prove that this setup gives the statements in
Theorem \ref{thm:kalman}, Theorem \ref{thm:functorial} and Theorem \ref{thm:differentcappings}.

\subsection{Proof of Theorem \ref{thm:kalman} and Theorem
  \ref{thm:differentcappings}}\label{sec:kalman-pf}
  
To prove Theorem \ref{thm:kalman}, we must show that 
\begin{equation}\label{eq:samma}
  \Phi_L \circ \partial_+ = \partial_- \circ \Phi_L.
\end{equation}
This uses similar techniques as the proof of Theorem
\ref{thm:georgios}. That is, we will use compactness results from
[\cite{compact}, Section 11.3] to pair up 2-level buildings that emanate from the
left hand side of \eqref{eq:samma} with
2-level buildings that emanate from the the right hand side. We perform
this pairing in a way so that each
pair can be interpreted as the boundary
of a 1-dimensional moduli space. Then we compute the difference
in orientations induced by gluings of such buildings, and argue that all
the signs cancel.

\begin{proof}[Proof of Theorem \ref{thm:kalman}.]
 Assume that $L\subset \R \times J^1(M)$ is an exact Lagrangian
 cobordism with cylindrical Legendrian ends $\Lambda_\pm$. Let $a$ be
 a Reeb chord of $\Lambda_+$, and let $u_0 \in \M_{s,\Lambda_+}(a, b_1\cdots
 b_m)$, $v_0 \in \M_L(a,c_1\cdots c_l)$ be rigid disks. Also pick
 rigid disks $v_i \in \M_{L}(b_i, b^i_1\cdots b^i_{m_i})$ for
 $i=1,\dotsc,m$, and $u_j \in \M_{s,\Lambda_-}(c_j,f_1\cdots f_k)$ for
 some $j \in\{1,\dotsc,l\}$, such that we can interpret $u_0
 \#_{i=1}^m v_i$ and $v_0 \# u_j$ as broken boundary components of a
 1-dimensional component $\M\subset \M_L$. In particular, we assume that
 \begin{equation*}
   b^1_1\cdots b^1_{m_1} b_1^2\cdots b^2_{m_2}\cdots b^m_1 \cdots b_{m_m}^m = c_1\cdots
   c_{j-1} f_1\cdots f_k  c_{j+1} \cdots c_l.
 \end{equation*}
  Let $\dbar_T$ denote the dbar-problem we get when we glue the fully
  capped $\dbar_{\hat u_0}$-problem to the fully capped $\dbar_{\hat
    v_i}$-problems, $i=1,\dotsc,m$. This gives the same problem as if
  we first glue $u_0$ to $v_1,\dotsc, v_m$, then glue the capping
  operators to the punctures of this new problem, and then glue
  this to the glued capping disks at $b_1,\dotsc,b_m$.

We also have to analyze the situation at the other boundary component of
$\M$. To that end, let $\dbar_{\tilde T}$ denote the dbar-problem that we get
if we glue the fully capped $\dbar_{\hat v_0}$-problem to the fully
capped $\dbar_{\hat u_j}$-problem. This is the same dbar-problem as we get
if we first glue $v_0$ to $u_j$, then glue the capping
  operators to the punctures of this new problem, and then glue this fully capped problem to the glued capping
disk at $c_j$.

Now, following the arguments in the proof of Lemma \ref{lma:lma3}, we first calculate the difference in
orientation of $\det \dbar_T$, induced by the 2 different gluings
described above, and then we do the same thing for $\det \dbar_{\tilde T}$. If
these differences cancel when we take the orientation of the
space of conformal variations associated to $\M$ into account, then we
get that $\Phi_L \circ \partial_+ = \partial_- \circ \Phi_L$.

From the gluings for $\dbar_T$ we get the following two exact
sequences. To simplify notation, we assume that $u_i$, $v_j$, $i=0,\dotsc,l$, $j=0,\dotsc,m$,
all have at least 2 negative punctures.

The first exact sequence has the form
\begin{equation}\label{eq:glu1}
  \kd_T \to
  \begin{bmatrix}
    \begin{pmatrix}
      \kd_{u_0}\\
      \kd_{a,+}\\
      \kd_{b_m,-}\\
      \vdots \\
      \kd_{b_1,-}
    \end{pmatrix}\\
    \begin{pmatrix}
      \kd_{v_m}\\
      \kd_{b_m,+}\\
      \Ker \capp(v_m,-)
    \end{pmatrix}\\
    \vdots\\
    \begin{pmatrix}
      \kd_{v_1}\\
      \kd_{b_1,+}\\
      \Ker \capp(v_1,-)
    \end{pmatrix}
  \end{bmatrix}
  \to
    \begin{bmatrix}
    \begin{pmatrix}
      \cd_{u_0}\\
      \cd_{a,+}\\
      \cd_{b_m,-}\\
      \vdots \\
      \cd_{b_1,-}
    \end{pmatrix}\\
    \R_{t,m}\\
    \R^{n+d_A}\\
    \begin{pmatrix}
      \cd_{v_m}\\
      \cd_{b_m,+}\\
      \Coker \capp(v_m,-)
    \end{pmatrix}\\
    \vdots\\
        \R_{t,1}\\
    \R^{n+d_A}\\
    \begin{pmatrix}
      \cd_{v_1}\\
      \cd_{b_1,+}\\
      \Coker \capp(v_1,-)
    \end{pmatrix}
  \end{bmatrix}
  \to
  \cd_T,
\end{equation}
where $\Ker \capp(v_i,-) = \hat \oplus_{j=1}^{m_i} \kd_{b^i_j,-}$, and
similar for the cokernel, and where we use the notation $\R_{t,i}$ to
indicate the cokernel which is born when we glue the $\hat v_i$-disk to
the bigger problem. 

Similarly, the second gluing sequence has the form
\begin{equation}
  \label{eq:glu2}
  \kd_T \to
  \begin{bmatrix}
    \begin{pmatrix}
      \kd_{u_0}\\
      \kd_{v_m}\\
      \vdots\\
      \kd_{v_1}\\
      \kd_{a,+}\\
      \Ker \capp(v_{m},-)\\
      \vdots\\
      \Ker \capp(v_1,-)
    \end{pmatrix}\\
    \begin{pmatrix}
      \kd_{b_m,+}\\
      \R_t\\
      \kd_{b_m,-}
    \end{pmatrix}\\
    \vdots\\
    \begin{pmatrix}
      \kd_{b_1,+}\\
      \R_t\\
      \kd_{b_1,-}
    \end{pmatrix}
  \end{bmatrix}
  \to
    \begin{bmatrix}
    \begin{pmatrix}
      \cd_{u_0}\\
      \R_{t,m}\\
      \cd_{v_m}\\
      \vdots\\
      \R_{t,1}\\
      \cd_{v_1}\\
      \cd_{a,+}\\
      \Coker \capp(v_{m},-)\\
      \vdots\\
      \Coker \capp(v_1,-)
    \end{pmatrix}\\
    \R_t\\
    \R^{n+d_A}\\
    \begin{pmatrix}
      \cd_{b_m,+}\\
      \cd_{b_m,-}
    \end{pmatrix}\\
    \vdots\\
    \R_t\\
    \R^{n+d_A}\\
    \begin{pmatrix}
      \cd_{b_1,+}\\
      \cd_{b_1,-}
    \end{pmatrix}
  \end{bmatrix}
  \to \cd_T.
\end{equation}
Again we have omitted the maps $0 \to \kd_T$ and $\cd_T \to 0$ in \eqref{eq:glu1} and \eqref{eq:glu2} to fit the sequences within page margins.

Perform rearrangements similar to those in the proof of Lemma \ref{lma:lma3}. That is, first notice  
 that by construction of the gluing map, each $\R_t$-summand in
the second column in \eqref{eq:glu2} is mapped by the identity to
the corresponding $\R_t$-summand in the third column. That is, the
$\R_t$-summand in the kernel of the glued capping disk at $b_i$ is mapped by
the identity to the $\R_t$-summand that corresponds to the cokernel
which is born when we glue the glued capping disk at $b_i$ to the larger
problem.  Thus,
using that we require the capping operators to satisfy \ref{c3} we
can remove all these occurrences of $\R_t$ from the sequence
\eqref{eq:glu2} without changing any induced orientation. This is done
by moving all these space to the bottom of column 2 and 3,
respectively, and recalling our orientation convention \eqref{eq:o2}.

Next we rearrange vector  spaces in the sequence \eqref{eq:glu1}, and we start with letting $\kd_{b_i,-}$ switch position with $\Ker \capp(v_i,-)$ for
$i=m,\dotsc,1$. By \ref{c3} we have that $\dim \Ker \capp(v_i,-) = m_i$, so that we get that this rearrangement costs $(-1)^{\sigma_1}$,
\begin{align*}
  \sigma_1 &= (m_m +1)\cdot(m-1) +m_m + \dotsc + (m_1+1)\cdot(m-1)+
             m_1\\
  &\equiv m \cdot \sum_{i=1}^m m_i + m(m+1) \equiv m \cdot \sum_{i=1}^m m_i \pmod{2}.
\end{align*}

Next we move $\R_{t,m}$ and  $\cd_{v_m}$  to the place just
below $\cd_{u_0}$. Using a calculation similar to \eqref{eq:forenk} we see that this costs $(-1)^\nu$,
\begin{align}\nonumber
  \nu &= \dim \R_{t,m}\cdot \dim \R^{n+d_A} + ( \dim \R_{t,m} + \dim \cd_{v_m})\\ \nonumber
  &\quad \cdot\left(\dim \R^{n+d_A} + \dim \left(\hbp_{i=1}^{m} \cd_{b_{i},-}\right) + \dim \cd_{a,+}\right)\\  \label{eq:sehar}
  &\equiv n + d_A + (1+m_m)\cdot(n+d_A+n+d_A) \equiv n +d_A \pmod{2}.
\end{align}
Perform the similar move with the other such spaces, so that we end up with
\begin{equation}
  \label{eq:cokernelar}
  \cd_{u_0} \oplus \R_{t,m} \oplus \cd_{v_m} \oplus \dotsc \oplus
  \R_{t,1} \oplus \cd_{v_1}
\end{equation}
at the top of the third column in the sequence \eqref{eq:glu1}. These
moves have a total cost $(-1)^{\sigma_2}$,
\begin{equation*}
  \sigma_2 = m \cdot(n+d_A)+ \sum_{i=1}^m(m-i)(m_i+1) \equiv m \cdot(n+d_A) + m \cdot \sum_{i=1}^m m_i + m + \sum_{i=1}^mi\cdot(m_i+1),
\end{equation*}
modulo 2.

Indeed, to move $\R_{t,i}$ and  $\cd_{v_{i}}$ costs $(-1)^{n+d_A+(m-i)(m_i+1)}$. Here $n +d_A$ is calculated just as in \eqref{eq:sehar}, and $(m-i)(m_i+1)$ comes from moving $\R_{t,i}\oplus \cd_{v_{i}}$ over $m-i$ copies of $\R^{n+d_A}$ and over the cokernels of the capping operators associated to $v_m,v_{m-1},\dotsc v_{i+1}$. Here we have used an equation similar to \eqref{eq:forenk}, namely that by  \ref{c3}, \eqref{eq:dim2} and again \ref{c3} we have that 
\begin{equation}\label{eq:forkob}
\dim \cd_{b_i,-} \equiv |b_i| \equiv \sum_{j=1}^{m_i}|b_j^i| \equiv \dim \Coker
\capp(v_i,-) \pmod{2}.     
\end{equation}

In the last step we switch position of $\cd_{b_i,-}$ and $\Coker
\capp(v_i,-)$ for $i=m,\dotsc,1$. Using \eqref{eq:forkob}, we see that this move costs $(-1)^{\sigma_3}$,
\begin{align*}
  \sigma_3 &=  \dim \left(\cd_{b_i,-} \oplus \Coker \capp(v_i,-) \right) \cdot E +  \dim \cd_{b_i,-} \cdot \dim \Coker \capp(v_i,-) \\ 
  &\equiv  0 \cdot E + \sum_{i=1}^m |b_i|^2 \equiv \sum_{i=1}^m |b_i|.
\end{align*}
Here $E$ is an expression involving dimensions of cokernels of capping operators and of a number of copies of $\R^{n+d_A}$. 

Thus, in total we get a permutation sign $(-1)^\sigma$,
\begin{align} \nonumber
  \sigma &= \sigma_1 + \sigma_2+ \sigma_3 \equiv m\cdot(n+d_A) + m +  \sum_{i=1}^mi\cdot(m_i+1) +  \sum_{i=1}^m
    |b_i|\\ 
         &\equiv \sum_{i=1}^mi\cdot(m_i+1) + \sum_{i=1}^m(|b_i|+n+d_A+
           1) \pmod{2}. \label{eq:sigsig}
\end{align}
(Recall that we assume that $m_i>1$ so that $\kd_{v_i}=0$ for $i=1,\dotsc,m$.)

Now we do the similar thing for the $\dbar_{\tilde T}$-problem, which gives a
total permutation sign $(-1)^{\tilde \sigma}$, where
\begin{equation}\label{eq:tildesigma}
  \tilde \sigma = (|c_j| + n +d_A+1) + j\cdot(k+1) + k + l + \sum_{i=1}^{j-1}|c_i|.
\end{equation}
Indeed, in this case we get the following two different gluing sequences 
\begin{equation}\label{eq:glu21}
  \kd_{\tilde T} \to
  \begin{bmatrix}
    \begin{pmatrix}
      \kd_{v_0}\\
      \kd_{a,+}\\
      \kd_{c_l,-}\\
      \vdots \\
      \kd_{c_j,-}\\
      \vdots \\
      \kd_{c_1,-}
    \end{pmatrix}\\
    \begin{pmatrix}
      \kd_{u_j}\\
      \kd_{c_j,+}\\
      \Ker \capp(u_j,-)
    \end{pmatrix}
  \end{bmatrix}
  \to
    \begin{bmatrix}
    \begin{pmatrix}
      \cd_{v_0}\\
      \cd_{a,+}\\
      \cd_{c_l,-}\\
      \vdots \\
      \cd_{c_j,-}\\
      \vdots \\
      \cd_{c_1,-}
    \end{pmatrix}\\
    \R_{t}\\
    \R^{n+d_A}\\
    \begin{pmatrix}
      \cd_{u_j}\\
      \cd_{c_j,+}\\
      \Coker \capp(u_j,-)
    \end{pmatrix}
    \end{bmatrix}
  \to
  \cd_{\tilde T},
\end{equation}
and
\begin{equation}
  \label{eq:glu22}
  \kd_{\tilde T} \to
  \begin{bmatrix}
    \begin{pmatrix}
      \kd_{v_0}\\
      \kd_{u_j}\\
      \kd_{a,+}\\
      \kd_{c_l,-}\\
      \vdots\\
      \kd_{c_{j+1},-}\\
      \Ker \capp(u_{j},-)\\
      \kd_{c_{j-1},-}\\
      \vdots\\
      \kd_{c_1,-}
      \end{pmatrix}\\
    \begin{pmatrix}
      \kd_{c_j,+}\\
      \R_t\\
      \kd_{c_j,-}
    \end{pmatrix}
  \end{bmatrix}
  \to
    \begin{bmatrix}
    \begin{pmatrix}
      \cd_{v_0}\\
      \R_{t}\\
      \cd_{u_j}\\
      \cd_{a,+}\\
      \cd_{c_l,-}\\
      \vdots\\
      \cd_{c_{j+1},-}\\
      \Coker \capp(u_{j},-)\\
      \cd_{c_{j-1},-}\\
      \vdots\\
      \cd_{c_1,-}
    \end{pmatrix}\\
    \R_t\\
    \R^{n+d_A}\\
    \begin{pmatrix}
      \cd_{c_j,+}\\
      \cd_{c_j,-}
    \end{pmatrix}
  \end{bmatrix}
  \to \cd_{\tilde T}.
\end{equation}
Here $\Ker \capp(u_j,-) = \hat \oplus_{i=1}^{k} \kd_{f_i,-}$, and
similar for the cokernel. Again we have omitted the maps $0 \to \kd_{\tilde T}$ and $\cd_{\tilde T} \to 0$  to fit the sequences within page margins.  

To compare these two sequences we do the usual rearrangements. First consider the sequence \eqref{eq:glu22}, and move the $\R_t$-summand in the second column and the lowest $\R_t$-summand in the third column to the bottom. This costs nothing, just like in the case for the $\dbar_T$-problem.

Next consider the sequence \eqref{eq:glu21}. Begin by moving $\kd_{u_j}$ to the place just below $\kd_{v_0}$. Again we assume that both disks have more than one negative puncture, so that $\kd_{u_j}$ is $1$-dimensional. Hence this costs $(-1)^{\sigma_0}$,
\begin{equation*}
 \sigma_0 = l.
\end{equation*}
Here we have as usual used \ref{c3}. 

Then switch places of $\kd_{c_j,-}$ and $\Ker \capp(u_j,-)$, which costs  $(-1)^{\sigma_1}$,
\begin{equation*}
 \sigma_1 = kj + j + 1.
\end{equation*}
This is calculated similar to how $\sigma_1$ was calculated in the $\dbar_T$-case.

Next we move $\R_t \oplus \cd_{u_j}$ to the place just below $\cd_{v_0}$. This costs $(-1)^{\sigma_2}$,
\begin{equation*}
 \sigma_2 = n + d_A + k+1,
\end{equation*}
and is calculated similar to $\sigma_2$ in the $\dbar_T$-case. 

Finally we switch position of $\cd_{c_j,-}$ and $\Coker \capp (u_j,-)$, with a cost of  $(-1)^{\sigma_3}$,
\begin{equation*}
 \sigma_3 = 1+\sum_{i=1}^{j} |c_i|. 
\end{equation*}
The reason for the extra $1$ here compared with the $\dbar_T$-calculation of $\sigma_3$ is that we now have $\dim \cd_{c_j,-} \equiv |c_j| \equiv \sum_{i=1}^{k}|f_i|+1\equiv \dim \Coker
\capp(u_j,-)+1$, modulo $2$. 

We see that all this sum up to a total sign $(-1)^{\tilde \sigma}$, given in \eqref{eq:tildesigma}.

Now it remains to calculate the contribution from the orientation of
the space of conformal variations. To that end, let $\mu_i$ be the capping sign of $v_i$, $i=0,\dotsc, m$, and $\epsilon_i$ the capping sign of $u_i$, $i=0,j$. To finish the proof of the theorem we copy the arguments from the proof of Lemma  \ref{lma:lma3}. That is, we add together the signs we get from the maps in the commutative diagrams corresponding to the diagram \eqref{eq:twostar}. In the case of the  $\dbar_{u_j\# v_0 }$-problem this diagram now reads   
\begin{equation}\label{eq:twostartilde}
\xymatrix{
 T\Co_{v_0}\oplus T\Co_{u_j} \oplus \R \ar[d] \ar[r] & T\Co_{u_j\#v_0} \ar[d] \\
 \cd_{v_0} \oplus \cd_{u_j} \ar[r] & \cd_{u_j\#v_0}.
 }
\end{equation}
We assume that the vector spaces are oriented as usual; the cokernels are given their capping orientations, the spaces of conformal variations are given their fixed orientation from Section \ref{sec:confvar}, and the $\R$-summand is given orientation as the outward normal from the space $T\Co_{u_j \# v_0}$.

It follows that the left vertical map restricted to $T\Co_{v_0}\oplus T\Co_{u_j}$ is an isomorphism of sign 
\begin{equation*}
 \mu_0\epsilon_j,
\end{equation*}
and from Lemma \ref{lma:confvar} we get that the upper horizontal map is an isomorphism of sign
\begin{equation*}
 (-1)^{(k-1)j +1}.
\end{equation*}

It remains to compute the sign of 
 the lower horizontal map. We claim that this is an isomorphism of sign $(-1)^{\sigma_{\tilde T}}$
\begin{align*}
 \sigma_{\tilde T} = \tilde \sigma + \tilde \sigma_0+ k +(|c_j| + n +d_A+1). 
\end{align*}
Here $\tilde \sigma_0$ is the sign that corresponds to the sign $\sigma_0$ in \eqref{eq:cokersign} and will be computed below, $k$ comes from Definition \ref{def:capping} applied to the capping orientation of $u_j$, and $(|c_j| + n +d_A+1)$ comes from the fact that the orientation of the glued capped disk at $c_j$ is given by $(-1)^{(|c_j| + n +d_A+1)}$ times the canonical orientation, see Definition \ref{def:capor}.

Now we argue that  $\tilde \sigma_0$  is given by
\begin{equation*}
 \tilde \sigma_0 = \dim \cd_{u_j} = k.
\end{equation*}
Indeed, the gluing sequence that corresponds to \eqref{eq:ny1} now reads
\begin{equation*}
0 \to
   \kd_{u_j \# v_0}
\to 
\begin{bmatrix}
  \kd_{v_0} \\
  \kd_{u_j} \\
  \end{bmatrix}
\to
\begin{bmatrix}
  \cd_{v_0} \\
  \R_t \\
 \cd_{u_j} \\
\end{bmatrix}
\to
\cd_{u_j \# v_0}
\to 0,
\end{equation*}
which moreover can be simplified to 
\begin{equation}\label{eq:simpa}
0 \to
   \R_t
\to 
\begin{bmatrix}
  0\\
  \R_t 
  \end{bmatrix}
\to
\begin{bmatrix}
  \cd_{v_0} \\
  \R_t \\
 \cd_{u_j} \\
\end{bmatrix}
\to
\cd_{u_j \# v_0}
\to 0,
\end{equation}
and which in turn can be reduced to 
\begin{equation*}
0 \to
   0
\to 
  0\\
\to
\begin{bmatrix}
  \cd_{v_0} \\
 \cd_{u_j} \\
\end{bmatrix}
\to
\cd_{u_j \# v_0}
\to 0,
\end{equation*}
with a cost of $(-1)^{\sigma_0}$, using that the leftmost non-trivial map in \eqref{eq:simpa} can be seen as given by $t \mapsto (0,t)$, and the map between the third and fourth column can be seen as given by $s \mapsto (0,s,0)$.

So what we get is that 
\begin{align*}
 \sigma_{\tilde T} = \tilde \sigma + \tilde \sigma_0+ k +(|c_j| + n +d_A+1) \equiv  j\cdot(k+1) + l + k + \sum_{i=1}^{j-1}|c_i| \pmod{2}.
\end{align*}
Hence it follows that the orientation of the boundary component of $\M$ that corresponds to the $\dbar_{u_j\# v_0 }$-problem is given by
\begin{equation}\label{eq:signtilde}
 (-1)^{\sigma_{\tilde T }+ (k-1)j +1}\mu_0\epsilon_j = (-1)^{ 1+ l + k + \sum_{i=1}^{j-1}|c_i|}\mu_0\epsilon_j
\end{equation}
times the outward normal.

The situation for the $\dbar_T$-problem is slightly more involved. To simplify notation let $\dbar_h = \dbar_{v_1\#\dotsm \# v_m \#u_0}$. 

To compute the sign that corresponds to $\sigma_0$ in \eqref{eq:cokersign}, one should consider the iterated gluing sequence 
\begin{equation*}
0 \to
   \kd_{h}
\to 
\begin{bmatrix}
  \kd_{u_0} \\
  \kd_{v_m} \\
  \vdots\\
  \kd_{v_1} \\
  \end{bmatrix}
\to
\begin{bmatrix}
  \cd_{u_0} \\
  \R_{t,m}\\
  \cd_{v_m} \\
  \vdots\\
  \R_{t,1}\\
  \cd_{v_1} \\
  \end{bmatrix}
\to
\cd_h
\to 0,
\end{equation*}
which simplifies to 
\begin{equation}\label{eq:simpat}
0 \to
   0
\to 
\begin{bmatrix}
  \R_t \\
  0 \\
  \vdots\\
  0 \\
  \end{bmatrix}
\to
\begin{bmatrix}
  \cd_{u_0} \\
  \R_{t,m}\\
  \cd_{v_m} \\
  \vdots\\
  \R_{t,1}\\
  \cd_{v_1} \\
  \end{bmatrix}
\to
\cd_h
\to 0,
\end{equation}
and which we can reduce to 
\begin{equation*}
0 \to
   0
\to 
\begin{bmatrix}
  0 \\
  0 \\
  \vdots\\
  0 \\
  \end{bmatrix}
\to
\begin{bmatrix}
  \cd_{u_0} \\
  \cd_{v_m} \\
  \R_{t,m}\\
  \cd_{v_{m-1}} \\
  \vdots\\
  \R_{t,2}\\
  \cd_{v_1} \\
  \end{bmatrix}
\to
\cd_h
\to 0,
\end{equation*}
if we compensate by the sign $(-1)^{\nu_0}$,
\begin{equation*}
 \nu_0 = 1+\sum_{i=1}^m m_i.
\end{equation*}
Here $\sum_{i=1}^m m_i$ comes from moving the $\R_{t,i}$-summand over $\cd_{v_i}$ for $i=1,\dotsc, m$, and the $1$ comes from removing the $\R_t$ and the $\R_{t,1}$ summand together with the fact that the leftmost non-trivial map in \eqref{eq:simpat} can now be seen as given by $t \mapsto (0,-t,0,\dotsc,-t,0)$.   

To compute the orientation of the boundary component of $\M$ that corresponds to the $\dbar_h$-problem we consider the commuting diagram 
\begin{equation}\label{eq:trestar}
 \xymatrix{
 T\Co_{u_0} \oplus T\Co_{v_m} \oplus \R \oplus  \dotsc \oplus T\Co_{v_{1}} \oplus \R \ar[d] \ar[r] &  T\Co_{w\#u_0} \ar[d] \\
 \cd_{u_0} \oplus \cd_{v_m} \oplus \R \oplus \dotsc \oplus \cd_{v_2} \oplus \R \oplus \cd_{v_{1}} \ar[r] &  \cd_{w \#u_0}
 }
\end{equation}
Here $w =v_1\#\dotsm \# v_m$. 
We assume that the vector spaces are oriented as usual; the cokernels are given their capping orientations, the spaces of conformal variations are given their fixed orientation from Section \ref{sec:confvar}, and the $\R$-summand to the right of $T\Co_{v_{i}}$ is given orientation as the outward normal from the space $T\Co_{v_i\#\dotsm \# v_m \#u_0}$ for $i=1, \dotsc, m$. The $\R$-summands in the lower row are given orientations from $\R_t$.

The left vertical map is given by 
\begin{equation*}
(x,y_m,t_m,\dotsc,y_1, t_1) \mapsto (\bar\Psi_0(x),  \bar\Psi_m(y_m),t_m, \dotsc, t_{2},  \bar\Psi_1(y_1)).
\end{equation*}
Here $\bar\Psi_i$ is the map \eqref{eq:psibar} associated to $v_i$, $i=1, \dotsc,m$, and to $u_0$ for $i=0$.  Using the fact that the gluing cokernel which occurs in the $\R_t$-direction can be interpreted as the outward normal of the space of conformal variations corresponding to the glued problem,
we get that this map restricted to the complement of the last $\R$-summand is an isomorphism of sign 
\begin{equation*}
 \epsilon_0\mu_1\dotsm\mu_m.
\end{equation*}

The upper horizontal map is an isomorphism of sign $(-1)^{\nu_2}$,
\begin{equation*}
 \nu_2 = ((m_m-1) m +1)+((m_{m+1}-1) (m-1) +1) +\dotsc +
             ((m_1-1) 1 + 1 ) \equiv m + \sum_{i=1}^m i\cdot(m_i+1),
\end{equation*}
modulo 2. To see this, we inductively note that by Lemma \ref{lma:confvar} we have that the inclusion of $T\Co_{v_{i+1}\#\dotsm \# v_m \#u_0} \oplus T\Co_{v_i} \oplus \R $ into $T\Co_{v_i\#\dotsm \# v_m \#u_0}$ is an isomorphism of sign $(-1)^{(m_i-1)i + 1}$.

Finally, we have  that the lower horizontal map is an isomorphism of sign $(-1)^{\sigma_T}$, 
\begin{equation*}
 \sigma_T = \sigma + \nu_0 + m + \sum_{i=1}^m(|b_i|+n+d_A+
           1) \equiv \sum_{i=1}^mi\cdot(m_i+1) +1+\sum_{i=1}^m m_i + m.
\end{equation*}
Here $m$ comes from Definition \ref{def:capping} applied to the capping orientation of $u_0$, and $\sum_{i=1}^m(|b_i|+n+d_A+1)$ comes from the fact that the orientation of the glued capped disk at $b_i$ is given by $(-1)^{(|b_i| + n +d_A+1)}$ times the canonical orientation for $i = 1, \dotsc,m$. Recall that $\sigma$ was computed in \eqref{eq:sigsig}.

By the commutativity of the diagram \eqref{eq:trestar}, 
and since the right vertical map is the one we use to get the orientation of the one-dimensional moduli space, it follows that the orientation of the boundary component of $\M$ that corresponds to the $\dbar_h$-problem is given by
\begin{align*}\label{eq:signt}
 (-1)^{\sigma_{T}+ \nu_2}\epsilon_0\mu_1\dotsm\mu_m & =
 (-1)^{ \sum_{i=1}^m i\cdot(m_i+1) +1+\sum_{i=1}^m m_i + m + m + \sum_{i=1}^m i\cdot(m_i+1)} \epsilon_0\mu_1\dotsm\mu_m\\
 & = (-1)^{l+k } \epsilon_0\mu_1\dotsm\mu_m
\end{align*}
times the outward normal. Here we have used that $\sum_{i=1}^m m_i = l+k -1$. 

Since the orientation of $\M$ at the boundary component that corresponds to the $\dbar_h$-problem is opposite to the orientation at the boundary component of $\M$ that corresponds to the $\dbar_{u_j\# v_0 }$-problem, it follows that 
\begin{equation*}
 (-1)^{l+k } \epsilon_0\mu_1\dotsm\mu_m = - (-1)^{ 1+ l + k + \sum_{i=1}^{j-1}|c_i|}\mu_0\epsilon_j.
\end{equation*}
Since $(-1)^{ \sum_{i=1}^{j-1}|c_i|}$ comes from the Leibniz rule it follows that $ \Phi_L \circ \partial_+ = \partial_- \circ \Phi_L$, and we have proved the theorem.

\end{proof}

\begin{proof}[Proof of Theorem \ref{thm:differentcappings}]
  From Remark \ref{rmk:notv}, Lemma \ref{lma:lma3} and the proof of invariance in [\cite{orientbok}, Section 4.3] we get that there exists an admissible system $\Sy$
  of capping operators satisfying the first statement in the
  theorem. Now, if $\Sy'$ is another admissible system of capping
  operators, then lift both systems to the
  symplectization, as described in Section \ref{sec:orientcap}.

   In this setting, the result about the DGA-isomorphism follows by exactly the same arguments as above, with $L = \R
  \times \Lambda$ and where $\Lambda_+ = \Lambda$ is equipped with the
  system $\mathcal{S}$ and $\Lambda_- = \Lambda$ is equipped with the
  system $\mathcal{S'}$. The DGA-isomorphism is then given by
  \begin{equation*}
    \Phi_{\mathcal{S}, \mathcal{S'}}(a) = \sigma(u_a)a,
  \end{equation*}
  where $a$ is a Reeb chord of $\Lambda$, $u_a = \R \times a$, and
  $\sigma(u_a)$ is the capping sign of $u_a$ with respect to the induced
  system of $L$. 
  
Thus, any admissible system gives rise to a DGA which is isomorphic to
the system in \cite{orientbok}, and the result follows. 
  
\end{proof}
\subsection{Proof of Theorem \ref{thm:functorial}}
In this section we prove that $\Phi$ satisfies the functorial
properties stated in Theorem \ref{thm:functorial}. To that end, if $L
= \R \times \Lambda$ and we wish to prove (\ref{eq:funct2}), note that we shall assume that the capping
system $\Sy$ of $\Lambda_+ = \Lambda$ equals that of $\Lambda_-=
\Lambda$, which in turn equals the induced system of $L$.

To prove (\ref{eq:funct1}), we need to have the following
setup. Assume that $L_1$ is an exact Lagrangian cobordism from
$\Lambda_0$ to $\Lambda_1$, and that $L_2$ is an exact Lagrangian
cobordism from $\Lambda_1$ to $\Lambda_2$. Assume that $L_1$ and $L_2$
are equipped with spin structures such that the induced spin structure
on $\Lambda_1$ from $L_1$, regarded as the negative boundary of $L_1$,
equals the induced spin structure on $\Lambda_1$ induced from $L_2$,  regarded as the positive
boundary of $L_2$. Let $\Sy_i$ be fixed admissible systems of capping
operators for $\Lambda_i$, $i=0,1,2$, let $\Sy_{01}$ denote the
induced system on $L_1$, $\Sy_{12}$ the induced system on $L_2$, and
$\Sy_{02}$ the induced system on the concatenation $L_1\# L_2$.

\begin{proof}[Proof of Theorem \ref{thm:functorial}]
  We must prove that
  \begin{equation*}
    \Phi_{\R \times \Lambda, \Sy} = \id, \qquad \Phi_{L_2,\Sy_{12}}
    \circ \Phi_{L_1,\Sy_{01}} = \Phi_{L_1\#L_2, \Sy_{02}}.
  \end{equation*}

We prove the statement for the identity map. The statement for the
concatenation is similar to the proof of Theorem \ref{thm:kalman} and
details are left to the reader. Indeed, the only thing that has to be checked
in that case is the following. Let $u_0 \in \M_{L_1}(a, b_1\cdots b_m)$, $v_i \in
\M_{L_2}(b_i,b^i_1\cdots b^i_{m_i})
$, $i=1,\cdots, m$, $u_1 \in \M_{L_1\#L_2}(a, b_1^1 \cdots b_{m_1}^1\cdots b_1^m\cdots
b_{m_m}^m)$ be rigid disks. Then one must prove that the gluing
sequence for $\dbar_{\hat u_0} \#_{i=1}^m \dbar_{\hat v_i}$ induces the
same orientation on the total glued problem as the gluing sequence we
get when we glue the glued capping disks at $b_1,\dotsc,b_m$, with
respect to the $\Sy_1$-system, to the
dbar-problem $\dbar_{\hat u_1}$, capped off with the $\Sy_{02}$-system. But this follows from similar arguments as
in the proof of Theorem \ref{thm:kalman}.

We turn to the trivial cobordism. Let $a$ be a Reeb chord of
$\Lambda$, let $u = \R \times a$, and assume that $\Phi_{\R \times
  \Lambda, \Sy}(a) = (-1)^\sigma a$. Here the sign $(-1)^\sigma$
satisfies that $\Ker \dbar_u \simeq \R_t$ is given the capping orientation
$(-1)^\sigma \partial_t$ with respect to the system $\Sy$. We must prove
that $\sigma \equiv 0 \pmod{2}$.

Consider the situation when we concatenate two trivial
cobordisms. That is, let $L_1 = L_2=\R\times \Lambda$, let $u_1$
denote the disk $\R \times a$ of $L_1$, let $u_2$ denote the disk $\R
\times a$ of $L_2$ and let $u$ denote the disk $\R \times a$ in $L_1
\# L_2$. Let $\dbar^h_{a,+}$ denote the capping operator
associated to $a$, regarded as a Reeb chord of the positive end of
$L_1$, and let $\dbar^m_{a,-}$ denote the capping operator associated to
$a$, regarded as a Reeb chord of the negative end of $L_1$. Similarly, let
$\dbar^m_{a,+}, \dbar^l_{a,-}$ denote the capping operators of $a$
regarded as a Reeb chord of the positive and negative end of $L_2$, respectively.

From the concatenation we get the following exact sequence
\begin{equation}\label{eq:sista0}
0 \to  \kd_u \to
  \begin{bmatrix}
    \kd_{u_1}\\
    \kd_{u_2}
  \end{bmatrix}
  \to
  \R_t \to 0 \to 0,
\end{equation}
where all non-trivial spaces are isomorphic to $\R_t$. Notice that this sequence
induces $+\R_t$-orientation of $\kd_u$, given that $\kd_{u_i}$ is given
the orientation $(-1)^\sigma\R_t$, $i=1,2$.

Now, following our standard arguments, consider the exact gluing
sequence
\begin{equation}\label{eq:sista1}
  0\to \kd_T \to
  \begin{bmatrix}
    \begin{pmatrix}
      \kd_{u_1}\\
      \kd^h_{a,+}\\
      \kd^m_{a,-}
    \end{pmatrix}\\
    \begin{pmatrix}
      \kd_{u_2}\\
      \kd^m_{a,+}\\
      \kd^l_{a,-}
    \end{pmatrix}
  \end{bmatrix}
  \to
  \begin{bmatrix}
    \begin{pmatrix}
      \cd^h_{a,+}\\
      \cd^m_{a,-}
    \end{pmatrix}\\
    \R_t\\
    \R^{n+d_A}\\
    \begin{pmatrix}
      \cd^m_{a,+}\\
      \cd^l_{a,-}
    \end{pmatrix}\\
  \end{bmatrix}
  \to
\cd_T \to 0.
\end{equation}
By Lemma \ref{lma:kan}, this sequence induces the
canonical orientation of $\det \dbar_T$, where $T$ is the totally
glued problem, given that the capping operators are given their
capping orientation and $\kd_{u_i}$ is given the capping orientation $(-1)^\sigma
\R_t$, $i=1,2$. But, similar to our arguments in the proof of Theorem
\ref{thm:kalman}, the $\dbar_T$-problem can also be obtained from the
following gluing sequence
\begin{equation}\label{eq:sista2}
 0 \to \kd_T \to
  \begin{bmatrix}
    \begin{pmatrix}
      \kd_{u_1}\\
      \kd_{u_2}\\
      \kd^h_{a,+}\\
      \kd^l_{a,-}
    \end{pmatrix}\\
    \begin{pmatrix}
      \kd^m_{a,+}\\
      \R_t\\
      \kd^m_{a,-}
    \end{pmatrix}
  \end{bmatrix}
  \to
  \begin{bmatrix}
    \begin{pmatrix}
      \R_t\\
      \cd^h_{a,+}\\
      \cd^l_{a,-}
    \end{pmatrix}\\
    \R_t\\
    \R^{n+d_A}\\
    \begin{pmatrix}
      \cd^m_{a,+}\\
      \cd^m_{a,-}
    \end{pmatrix}\\
  \end{bmatrix}
  \to
\cd_T \to 0.
\end{equation}

We see that the orientation of $\det \dbar_T$ induced by this
sequence, where we assume that all spaces in columns 2 and 3 are
oriented as in sequence \eqref{eq:sista1}, equals the canonical orientation
times $(-1)^{\sigma_1}$,
\begin{equation}\label{eq:siggi}
 \sigma_1 = |a| + n+d_A+1.
\end{equation}
Indeed, this sign we get if we do the
usual rearrangements, comparing \eqref{eq:sista1} with
\eqref{eq:sista2}: we can remove the bottom-most $\R_t$ in columns 2 and
3 in \eqref{eq:sista2} without any cost, then move $\kd_{u_2}$ and the
gluing cokernel in \eqref{eq:sista1} to the corresponding position in
\eqref{eq:sista2} with the cost of $(-1)^{1+|a|+|a|+n+d_A+1}$, and then switch
positions of $\kd^m_{a,-}$ and $\kd^l_{a,-}$, and similar for the
cokernels, with a cost of $(-1)^{1+|a|}$.

Now we return to the original sequence \eqref{eq:sista2}. Using associativity of orientations under gluing (see [\cite{orientbok}, Section 3.2.3]),   
 we can use the sequence
\eqref{eq:sista0} to replace the pair $(\kd_{u_1}\oplus \kd_{u_2},
\R_t)$ in \eqref{eq:sista2} by the pair $(\kd_u, 0)$. We then get the exact sequence 
\begin{equation}\label{eq:sista3}
 0 \to \kd_T \to
  \begin{bmatrix}
    \begin{pmatrix}
      \kd_{u}\\
      \kd^h_{a,+}\\
      \kd^l_{a,-}
    \end{pmatrix}\\
    \begin{pmatrix}
      \kd^m_{a,+}\\
      \R_t\\
      \kd^m_{a,-}
    \end{pmatrix}
  \end{bmatrix}
  \to
  \begin{bmatrix}
    \begin{pmatrix}
      \cd^h_{a,+}\\
      \cd^l_{a,-}
    \end{pmatrix}\\
    \R_t\\
    \R^{n+d_A}\\
    \begin{pmatrix}
      \cd^m_{a,+}\\
      \cd^m_{a,-}
    \end{pmatrix}\\
  \end{bmatrix}
  \to
\cd_T \to 0.
\end{equation}
Moreover, from [\cite{orientbok}, Lemma 3.7] we get that this sequence induces the same orientation on $\det \dbar_T$ as the sequence \eqref{eq:sista2} does. That is, \eqref{eq:sista3} induces the canonical orientation times $(-1)^{\sigma_1}$, given that the capping operators are given their capping orientation and that $\kd_u$ is given the orientation induced by the sequence \eqref{eq:sista0}. Recall that the latter orientation is given by $\R_t$, which differs from the capping orientation of $\kd_u$ by a sign $(-1)^\sigma$. 

Now we change the orientation of $\kd_u$ to its capping orientation $(-1)^\sigma\R_t$, keeping the capping orientation of the capping operators. Then the orientation of $\det \dbar_T$ induced by  \eqref{eq:sista3} should change by a sign $(-1)^\sigma$, to be given by $(-1)^{\sigma_2}$ times the canonical orientation of $\det \dbar_T$, where 
\begin{equation}\label{eq:siggi2}
 \sigma_2 = \sigma + \sigma_1 = \sigma+ n + d_A +|a|+1.
\end{equation}
However, by our standard arguments it follows that this orientation should be given by $(-1)^{\sigma_3}$ times the canonical orientation of $\det \dbar_T$, where 
\begin{equation}\label{eq:siggi4}
 \sigma_3 = n + d_A +|a|+1.
\end{equation}
This follows from Lemma \ref{lma:kan} together with the fact that
 the upper part of \eqref{eq:sista3} corresponds to the capping sequence for $\dbar_u$ which by assumption is given the canonical orientation, and the lower part of \eqref{eq:sista3} corresponds to the gluing sequence for the capping disk at $a$ which by the assumptions together with Definition \ref{def:capor} is given the canonical orientation times  $(-1)^{ n + d_A +|a|+1}$.

Comparing \eqref{eq:siggi2} with \eqref{eq:siggi4} we get that we should have
\begin{equation*}
 (-1)^{\sigma+ n + d_A +|a|+1} = (-1)^{ n + d_A +|a|+1},
\end{equation*}
implying that $(-1)^{\sigma} = 1$. 
 This
concludes the proof of 
$\Phi_{\R \times \Lambda, \Sy} = \id$.  
\end{proof}

\subsection{A remark on orientation conventions}\label{sec:clock}
 The reason that we choose the convention \eqref{eq:gluconv} instead of the convention  
 \begin{equation*}
 0 \to \kd_{A\#B}
 \xrightarrow{\alpha}
 \begin{bmatrix}
  \kd_A \\
  \kd_B
 \end{bmatrix}
 \xrightarrow{\beta}
\begin{bmatrix}
 \cd_A \\
 \cd_B
\end{bmatrix}
 \xrightarrow{\gamma}
\cd_{A\#B} \to 0
\end{equation*}
from \cite{orientbok}, is that our choice seems to simplify the calculations in the proofs. That is, it is easier to have the $\R_t$-summand of the kernel of the $\dbar_u$-operator, for $u$ a pseudo-holomorphic disks of $\R \times \Lambda$, as close to the top of the columns in the gluing sequences as possible. 

In the same time as we want to use this simplification, we also want to take advantage of the calculations in the proofs of [\cite{orientbok}, Lemma 4.9 and Lemma 4.11]. This forces us to have all exact sequences (except for the ones for the glued capping disks) from \cite{orientbok} ''mirrored'' in the horizontal axis. For example, we have to glue the capping operators in the clockwise direction, starting at the positive puncture, and we also have to choose the orientations of the spaces of conformal variations to be ''opposite'' to the one in \cite{orientbok}. 

The reason that we keep the same convention for the gluing sequences for the glued capping disks as in \cite{orientbok} (these are not mirrored!) is to get rid of the sign $(-1)^{(n-1)(|a|+1)}$ in Remark \ref{rmk:dumsign}. This also has to do with the sign in Definition \ref{def:capor}. However, the reason for this sign is mostly to get the calculations in the proof of Theorem \ref{thm:kalman} and Theorem \ref{thm:functorial}  to work out well.

\section{Morse flow trees and abstract perturbations}\label{sec:trees}
In \cite{tobhonda} the techniques of abstractly perturbed flow trees are used to give explicit descriptions of DGA-morphisms $\Phi_L$ associated to elementary Legendrian isotopies, with coefficients in $\Z_2$. Here we explain how this can be done also with integer coefficients.

Let $L \subset \R \times J^1(M)$ be an exact Lagrangian cobordism. To this cobordism we associate a \emph{Morse cobordism} $L^{MO}$, which will be exact Lagrangian isotopic to $L$ relative the ends. See [\cite{tobhonda}, Section 2]. The advantage of considering $L^{MO}$ instead of $L$ is that we can use the Morse flow tree techniques from \cite{trees} to define $\Phi_{L^{MO}}$ to be given by a count of rigid flow trees instead of rigid disks.

For a definition of Morse flow trees we refer to [\cite{trees}, Section 2.2] and [\cite{tobhonda}, Section 4]. Briefly, these  trees are built out of flow lines of local gradient differences associated to $L^{MO}$, and there is a one-to-one correspondence between the rigid Morse flow trees and the rigid pseudo-holomorphic disks of $L^{MO}$. Thus, if we let $\M_{T, L}(a, {\bf b})$ denote the moduli space of Morse flow trees of $L^{MO}$ with positive puncture $a$ and negative punctures ${\bf b}$, it follows from [\cite{tobhonda}, Lemma 5.12] that $\Phi_L$ can be given by
\begin{equation*}
 \Phi_L(a) = \sum_{\dim \M_{T, L}(a, {\bf b}) =0}  |\M_{T, L}(a, {\bf b})| {\bf b}.
\end{equation*}
Here $ |\M_{T, L}(a, {\bf b})|$ denotes the algebraic count of
elements in the moduli space, where we have used the oriented
identification of $\M_{T, L}(a, {\bf b})$ and $\M_L(a,{\bf b})$ from
[\cite{orienttrees}, Theorem 1.1]. In summary, the orientation of a tree $\Gamma \in
\M_{T,L}(a,{\bf b})$ is defined by considering the cotangent lift of the
tree, which gives rise to a Lagrangian boundary condition for a corresponding punctured disk, and we get an associated linearized $\dbar$-operator $\dbar_\Gamma$. Moreover, we can glue  capping operators associated to the punctures of $\Gamma$ to this operator, to get a corresponding fully capped problem $\dbar_{\hat \Gamma}$ on the non-punctured disk. This gives rise to exactly the same gluing sequences as in the case of true $J$-holomorphic disks, and the \emph{capping orientation of $\Gamma$} is defined completely analogous to how it is done for $J$-holomorphic disks. For details we refer to [\cite{orienttrees}, Section 4].

To get the explicit formulas for the DGA-maps in \cite{tobhonda}, the trees of $L^{MO}$ are perturbed. See [\cite{tobhonda}, Section 6.3]. This is first done by a geometric perturbation, which is a perturbation of $L^{MO}$ together with a perturbation of the Riemannian metric. If we extend the orientation scheme of $\M_{T,L}$ to also be defined for the geometrically perturbed trees, it follows by straightforward arguments that for a generic geometric perturbation, the algebraic count of trees in $\M_{T,L}$ will be equal to the algebraic count of rigid, geometrically perturbed trees. We let $\Phi_{L,g}$ denote the 
DGA-map defined by the count of the rigid geometric perturbed trees. 

Next we consider an abstract perturbation of the trees, as defined in [\cite{tobkal}, Section 3.3 -- 3.4], compare also with [\cite{tobhonda}, Section 6.3.1]. In [\cite{tobhonda}, Lemma 6.4] it is proven that the induced map $\Phi_{L,a}$, given by a count of rigid, abstractly perturbed flow trees, is a DGA-morphism which is chain homotopic to $\Phi_{L,g}$, given that we are using $\Z_2$-coefficients. We claim that this chain homotopy can be lifted to $\Z$-coefficients.

Indeed, the proof of [\cite{tobhonda}, Lemma 6.4]  makes use of a 1-parameter family of abstract perturbations, starting at the geometric perturbation (which we can interpret as an abstract perturbation) and ending at the desired abstract perturbation. The chain homotopy is then defined by a  count of certain Morse flow trees of $L^{MO} \times D$, where $D$ is the unit disk and where the trees are induced by the  1-parameter family of perturbations.

From [\cite{tobkal}, Section 3.4] it follows that the boundary conditions of the abstractly perturbed trees are close to being boundary conditions for true trees. Hence we can extend the orientation scheme for flow trees to also include abstractly perturbed trees, both for the trees occurring in the formula for $\Phi_{L,a}$ and also for the flow trees in the chain homotopy just described. Thus it follows that the algebraic count of abstractly perturbed trees given by $\Phi_{L,a}$ is signed chain homotopic to the algebraic count of geometrically perturbed trees given by $\Phi_{L,g}$. 
It follows that $\Phi_{L,a}$ is chain homotopic to $\Phi_L$ over $\Z$.

\bibliographystyle{halpha}
\bibliography{c_main} 

\newcommand{\etalchar}[1]{$^{#1}$}
\def\cprime{$'$}
\begin{thebibliography}{CDRGG15}

\bibitem[BEH{\etalchar{+}}03]{compact}
F.~Bourgeois, Y.~Eliashberg, H.~Hofer, K.~Wysocki, and E.~Zehnder.
\newblock Compactness results in symplectic field theory.
\newblock {\em Geom. Topol.}, 7:799--888, 2003.

\bibitem[CDRGG]{floerlag2}
Baptiste Chantraine, Georgios Dimitroglou~Rizell, Paolo Ghiggini, and Roman
  Golovko.
\newblock Floer theory for {L}agrangian cobordisms.
\newblock arXiv:1511.09471.

\bibitem[CDRGG15]{floerlag}
Baptiste Chantraine, Georgios Dimitroglou~Rizell, Paolo Ghiggini, and Roman
  Golovko.
\newblock Floer homology and {L}agrangian concordance.
\newblock In {\em Proceedings of the {G}\"okova {G}eometry-{T}opology
  {C}onference 2014}, pages 76--113. G\"okova Geometry/Topology Conference
  (GGT), G\"okova, 2015.

\bibitem[Che02]{chekanov}
Yuri Chekanov.
\newblock Differential algebra of {L}egendrian links.
\newblock {\em Invent. Math.}, 150(3):441--483, 2002.

\bibitem[CM]{rm}
Roger Casals and Emmy Murphy.
\newblock Legendrian {F}ronts for {A}ffine {V}arieties.
\newblock arXiv:1610.06977.

\bibitem[DR16]{Georgios}
Georgios Dimitroglou~Rizell.
\newblock Lifting pseudo-holomorphic polygons to the symplectisation of
  {$P\times\Bbb{R}$} and applications.
\newblock {\em Quantum Topol.}, 7(1):29--105, 2016.

\bibitem[DRG14]{georgios-roman}
Georgios Dimitroglou~Rizell and Roman Golovko.
\newblock On homological rigidity and flexibility of exact {L}agrangian
  endocobordisms.
\newblock {\em Internat. J. Math.}, 25(10):1450098, 24, 2014.

\bibitem[EES05a]{legsub}
Tobias Ekholm, John Etnyre, and Michael Sullivan.
\newblock The contact homology of {L}egendrian submanifolds in {${\mathbb
  {R}}^{2n+1}$}.
\newblock {\em J. Differential Geom.}, 71(2):177--305, 2005.

\bibitem[EES05b]{orientbok}
Tobias Ekholm, John Etnyre, and Michael Sullivan.
\newblock Orientations in {L}egendrian contact homology and exact {L}agrangian
  immersions.
\newblock {\em Internat. J. Math.}, 16(5):453--532, 2005.

\bibitem[EES07]{pr}
Tobias Ekholm, John Etnyre, and Michael Sullivan.
\newblock Legendrian contact homology in {$P\times\mathbb{R}$}.
\newblock {\em Trans. Amer. Math. Soc.}, 359(7):3301--3335 (electronic), 2007.

\bibitem[EGH00]{sft}
Y.~Eliashberg, A.~Givental, and H.~Hofer.
\newblock Introduction to symplectic field theory.
\newblock {\em Geom. Funct. Anal.}, (Special Volume, Part II):560--673, 2000.
\newblock GAFA 2000 (Tel Aviv, 1999).

\bibitem[EHK16]{tobhonda}
Tobias Ekholm, Ko~Honda, and Tam\'as K\'alm\'an.
\newblock Legendrian knots and exact {L}agrangian cobordisms.
\newblock {\em J. Eur. Math. Soc. (JEMS)}, 18(11):2627--2689, 2016.

\bibitem[EK08]{tobkal}
Tobias Ekholm and Tam{\'a}s K{\'a}lm{\'a}n.
\newblock Isotopies of {L}egendrian 1-knots and {L}egendrian 2-tori.
\newblock {\em J. Symplectic Geom.}, 6(4):407--460, 2008.

\bibitem[Ekh07]{trees}
Tobias Ekholm.
\newblock Morse flow trees and {L}egendrian contact homology in 1-jet spaces.
\newblock {\em Geom. Topol.}, 11:1083--1224, 2007.

\bibitem[Ekh08]{ratsft}
Tobias Ekholm.
\newblock Rational symplectic field theory over {$\mathbb{Z}_2$} for exact
  {L}agrangian cobordisms.
\newblock {\em J. Eur. Math. Soc. (JEMS)}, 10(3):641--704, 2008.

\bibitem[Ekh16]{toblose}
Tobias Ekholm.
\newblock Non-loose {L}egendrian spheres with trivial contact homology {DGA}.
\newblock {\em J. Topol.}, 9(3):826--848, 2016.

\bibitem[EL]{el}
Tobias Ekholm and Yanki Lekili.
\newblock Duality between {L}agrangian and {L}egendrian invariants.
\newblock arXiv:1701.01284.

\bibitem[Eli98]{eliinv}
Yakov Eliashberg.
\newblock Invariants in contact topology.
\newblock In {\em Proceedings of the {I}nternational {C}ongress of
  {M}athematicians, {V}ol. {II} ({B}erlin, 1998)}, number Extra Vol. II, pages
  327--338, 1998.

\bibitem[ENS02]{sabloffng}
John~B. Etnyre, Lenhard~L. Ng, and Joshua~M. Sabloff.
\newblock Invariants of {L}egendrian knots and coherent orientations.
\newblock {\em J. Symplectic Geom.}, 1(2):321--367, 2002.

\bibitem[FH93]{Floer-Hofer}
A.~Floer and H.~Hofer.
\newblock Coherent orientations for periodic orbit problems in symplectic
  geometry.
\newblock {\em Math. Z.}, 212(1):13--38, 1993.

\bibitem[FOOO09]{fooo}
Kenji Fukaya, Yong-Geun Oh, Hiroshi Ohta, and Kaoru Ono.
\newblock {\em Lagrangian intersection {F}loer theory: anomaly and obstruction.
  {P}art {II}}, volume~46 of {\em AMS/IP Studies in Advanced Mathematics}.
\newblock American Mathematical Society, Providence, RI, 2009.

\bibitem[Kar]{orienttrees}
Cecilia Karlsson.
\newblock Orientations of {M}orse flow trees in {L}egendrian contact homology.
\newblock arXiv:1601.07346.

\bibitem[MS12]{jholo}
Dusa McDuff and Dietmar Salamon.
\newblock {\em {$J$}-holomorphic curves and symplectic topology}, volume~52 of
  {\em American Mathematical Society Colloquium Publications}.
\newblock American Mathematical Society, Providence, RI, second edition, 2012.

\bibitem[Zin16]{z}
Aleksey Zinger.
\newblock The determinant line bundle for {F}redholm operators: construction,
  properties, and classification.
\newblock {\em Math. Scand.}, 118(2):203--268, 2016.

\end{thebibliography}

\end{document}